\documentclass[a4paper,fleqn]{cas-sc}
\usepackage[numbers]{natbib}
\usepackage[english,ruled]{algorithm2e}
\def\tsc#1{\csdef{#1}{\textsc{\lowercase{#1}}\xspace}}
\tsc{WGM}
\tsc{QE}
\tsc{EP}
\tsc{PMS}
\tsc{BEC}
\tsc{DE}
\begin{document}

\newtheorem{thm}{Theorem}[section]
\newtheorem{lem}[thm]{Lemma}
\newtheorem{prop}[thm]{Proposition}
\newtheorem{corollary}[thm]{Corollary}
\newtheorem{ex}{Example}[section]
\newtheorem{defi}{Definition}[section]
\newenvironment{Assumptions}
{%
	\setcounter{enumi}{0}
	\renewcommand{\theenumi}{(\textbf{A}.\arabic{enumi})}
	\renewcommand{\labelenumi}{\theenumi}
	\begin{enumerate}}%
	{\end{enumerate} }
\newenvironment{Definitions}
{%
	\setcounter{enumi}{0}
	\renewcommand{\theenumi}{(\textbf{D}.\arabic{enumi})}
	\renewcommand{\labelenumi}{\theenumi}
	\begin{enumerate}}%
	{\end{enumerate} }

\newcommand{\norm}[1]{\ensuremath{\left\|#1\right\|}}
\newcommand{\abs}[1]{\ensuremath{\left|#1\right|}}
\newcommand{\Om}{\ensuremath{\Omega}}
\newcommand{\disp}{\ensuremath{\displaystyle}}

\newcommand{\cF}{\ensuremath{\mathcal{F}}}
\newcommand{\cN}{\ensuremath{\mathcal{N}}}
\newcommand{\cFt}{\ensuremath{\cF_t}}
\newcommand{\tcFt}{\ensuremath{\tilde\cF_t}}
\newcommand{\cS}{\ensuremath{\mathcal{S}}}
\newcommand{\cP}{\ensuremath{\mathcal{P}}}
\newcommand{\cB}{\ensuremath{\mathcal{B}}}
\newcommand{\bB}{\ensuremath{\Bbb{B}}}
\newcommand{\U}{\ensuremath{\Bbb{U}}}
\newcommand{\bX}{\ensuremath{\Bbb{X}}}
\newcommand{\vphi}{\ensuremath{\varphi}}
\newcommand{\bK}{\ensuremath{\Bbb{K}}}
\newcommand{\En}{\ensuremath{\mathbf{1}}}
\newcommand{\cL}{\ensuremath{\mathcal{L}}}
\newcommand{\Span}{\ensuremath{\mathrm{Span}\!}}
\newcommand{\ipen}[2]{\left\langle #1, #2\right\rangle_{(H^1)^\star,H^1}}
\newcommand{\iphto}[2]{\left\langle #1, #2\right\rangle_{(H^2_N)^\star,H^2_N}}
\newcommand{\overn}{\overset{n\uparrow\infty}{\longrightarrow}}
\newcommand{\ip}[2]{\left\langle #1, #2\right\rangle}
\newcommand{\cH}{\ensuremath{\mathcal{H}}}
\newcommand{\cA}{\ensuremath{\mathcal{A}}}
\newcommand{\cZ}{\ensuremath{\mathcal{Z}}}
\newcommand{\C}{\ensuremath{\mathbf{C}}}
\newcommand{\weak}{\rightharpoonup}
\newcommand{\cX}{\ensuremath{\mathcal{X}}}
\newcommand{\tcFs}{\ensuremath{\tilde\cF_s}}
\newcommand{\cFs}{\ensuremath{\cF_s}}
\newcommand{\weakstar}{\overset{\star}{\rightharpoonup}}

\newcommand{\Gw}{\ensuremath{\Gamma_{w}}}
\newcommand{\Gs}{\ensuremath{\Gamma_{y}}}
\newcommand{\Ga}{\ensuremath{\Gamma_{0}}}
\newcommand{\Gb}{\ensuremath{\Gamma_{1}}}
\newcommand{\uv}{\ensuremath{\underline{v}}}
\newcommand{\oG}{\ensuremath{\overline{\gamma}}}
\newcommand{\Do}{\ensuremath{D}}
\newcommand{\DoGs}{\ensuremath{\Do\bigcup\Gs}}
\newcommand{\cDo}{\ensuremath{\overline{\Do}}}
\newcommand{\DoT}{\ensuremath{(0,T)\times\Do}}
\newcommand{\DoTGs}{\ensuremath{(0,T)\times\left(\Do\bigcup\Gs\right)}}
\newcommand{\cDoT}{\ensuremath{(0,T)\times\overline{\Do}}}
\newcommand{\DOT}{\ensuremath{[0,T]\times\Do}}
\newcommand{\cDOT}{\ensuremath{[0,T]\times\overline{\Do}}}
\newcommand{\Oset}{\ensuremath{\mathcal{O}}}
\newcommand{\OT}{\ensuremath{\Oset_T}}
\newcommand{\K}{\ensuremath{\mathcal{K}}}
\newcommand{\A}{\ensuremath{\mathcal{A}}}

\newcommand{\Jet}{\ensuremath{\mathcal{P}}}
\newcommand{\cJet}{\ensuremath{\overline{\mathcal{P}}}}
\newcommand{\Eset}{\ensuremath{\mathcal{E}}}
\newcommand{\Kset}{\ensuremath{\mathcal{K}}}
\newcommand{\oKset}{\ensuremath{\overline{\mathcal{K}}}}
\newcommand{\Ham}{\ensuremath{\mathcal{H}}}
\newcommand{\Gop}{\ensuremath{\mathcal{G}}}
\newcommand{\Lop}{\ensuremath{\mathcal{L}}}
\newcommand{\Iop}{\ensuremath{\mathcal{I}}}
\newcommand{\Iopsub}{\ensuremath{\mathcal{I}_{\kappa}}}
\newcommand{\Iopsup}{\ensuremath{\mathcal{I}^{\kappa}}}

\newcommand{\D}{\ensuremath{\mathcal{D}}}
\newcommand{\Dp}{\ensuremath{\D^\prime}}
\newcommand{\ov}{\ensuremath{\overline{v}}}
\newcommand{\thh}{\ensuremath{\tilde{h}}}
\newcommand{\ta}{\ensuremath{\tilde{a}}}
\newcommand{\F}{\ensuremath{\mathcal{F}}}
\newcommand{\B}{\ensuremath{\mathcal{B}}}
\newcommand{\W}{\ensuremath{\mathcal{W}}}
\newcommand{\eps}{\ensuremath{\varepsilon}}
\newcommand{\R}{\ensuremath{\Bbb{R}}}
\newcommand{\Sy}{\ensuremath{\Bbb{S}}}
\newcommand{\E}{\ensuremath{\Bbb{E}}}
\newcommand{\pmin}{\ensuremath{\underline{\pi}}}
\newcommand{\pmax}{\ensuremath{\overline{\pi}}}
\newcommand{\ximin}{\ensuremath{\underline{\xi}}}
\newcommand{\ximax}{\ensuremath{\overline{\xi}}}
\newcommand{\tmax}{\ensuremath{t_{\mathrm{max}}}}
\newcommand{\xmax}{\ensuremath{x_{\mathrm{max}}}}
\newcommand{\loc}{\mathrm{loc}}
\newcommand{\meas}{\mathrm{meas}}
\newcommand{\Lenloc}{L_{\mathrm{loc}}^1}
\newcommand{\Div}{\mathrm{div}\,}
\newcommand{\Grad}{\mathrm{\nabla}}
\newcommand{\bA}{\mathbf{A}}

\newcommand{\tH}{\ensuremath{{H}^1}}

\newcommand{\sgn}{\mathrm{sign}}
\newcommand{\op}{\ensuremath{\overline{p}}}
\newcommand{\vv}{\ensuremath{\overline{v}}}
\newcommand{\uu}{\ensuremath{\overline{u}}}
\newcommand{\oq}{\ensuremath{\overline{q}}}
\newcommand{\opi}{\ensuremath{\overline{p}_i}}
\newcommand{\ops}{\ensuremath{\op^\star}}
\newcommand{\ue}{\ensuremath{u_\eps}}
\newcommand{\uei}{\ensuremath{u_{i,\eps}}}

\newcommand{\Iap}{I_{\mathrm{app}}}
\newcommand{\Iape}{I_{\mathrm{app},\eps}}
\newcommand{\Iapn}{I_{\mathrm{app},n}}
\newcommand{\Iapa}{I_{\mathrm{app},1}}
\newcommand{\Iapb}{I_{\mathrm{app},2}}
\newcommand{\Iapk}{I_{\mathrm{app},k}}
\newcommand{\Iapnk}{I_{\mathrm{app},n}^\kappa}

\newcommand{\Lsloc}{\ensuremath{L_{\mathrm{loc}}^s}}
\newcommand{\Linf}{\ensuremath{L^\infty}}
\newcommand{\Linfloc}{\ensuremath{L^\infty_{\mathrm{loc}}}}
\newcommand{\Lpla}{\ensuremath{L^{p_{1,l}}}}
\newcommand{\Lplb}{\ensuremath{L^{p_{2,l}}}}
\newcommand{\Lqla}{\ensuremath{L^{q_{1,l}}}}
\newcommand{\Lqlb}{\ensuremath{L^{q_{2,l}}}}
\newcommand{\Lplc}{\ensuremath{L^{p_{i,l}}}}
\newcommand{\Lqlc}{\ensuremath{L^{q_{i,l}}}}

\newcommand{\dW}{\ensuremath{dW}}

\newcommand{\oqs}{\ensuremath{\oq^\star}}
\newcommand{\Loqs}{\ensuremath{L^{\oqs}}}
\newcommand{\vpg}{\ensuremath{\varphi_\gamma}}
\newcommand{\vpgb}{\ensuremath{\varphi_{\gamma,\beta}}}

\newcommand{\Sspace}{\ensuremath{W^{1,q_{l}}_{\loc}(\R^N)}}
\newcommand{\Wpla}{\ensuremath{W^{1,p_{1,l}}}}
\newcommand{\Wqla}{\ensuremath{W^{1,q_{1,l}}}}
\newcommand{\Wplb}{\ensuremath{W^{1,p_{2,l}}}}
\newcommand{\Wqlb}{\ensuremath{W^{1,q_{2,l}}}}
\newcommand{\Wplc}{\ensuremath{W^{1,p_{i,l}}}}
\newcommand{\Wqlc}{\ensuremath{W^{1,q_{i,l}}}}

\newcommand{\CMJ}{\ensuremath{\mathcal{M}_j}}
\newcommand{\CMJl}{\ensuremath{\mathcal{M}_{j,l}}}
\newcommand{\CMJa}{\ensuremath{\mathcal{M}_{j,1}}}
\newcommand{\CMJb}{\ensuremath{\mathcal{M}_{j,2}}}
\newcommand{\CMJc}{\ensuremath{\mathcal{M}_{j,3}}}

\newcommand{\CMI}{\ensuremath{\mathcal{M}_i}}
\newcommand{\CME}{\ensuremath{\mathcal{M}_e}}
\newcommand{\Set}[1]{\ensuremath{\left\{#1\right\}}}

\newcommand{\pt}{\ensuremath{\partial_t}}
\newcommand{\pxl}{\ensuremath{\partial_{x_l}}}

\newcommand{\ailtuepa}
{\ensuremath{{a}_{1,l}\left(t,x,\frac{\partial
			u_{1,\eps}}{\partial x_l})}}
\newcommand{\tailtuepa}
{\ensuremath{\tilde{a}_{1,l}\left(t,x,\frac{\partial
			u_{1,\eps}}{\partial x_l}\right)}}
\newcommand{\tailtuatg}
{\ensuremath{\tilde{a}_{1,l}\left(t,x,\frac{\partial
			T_\gamma(u_1)}{\partial x_l}\right)}}
\newcommand{\tailtuepatm}
{\ensuremath{\tilde{a}_{1,l}\left(t,x,\frac{\partial
			T_M(u_{1,\eps})}{\partial x_l}\right)}}
\newcommand{\tailtuepatg}
{\ensuremath{\tilde{a}_{1,l}\left(t,x,\frac{\partial
			T_\gamma(u_{1,\eps})}{\partial x_l}\right)}}
\newcommand{\ailtuepb}
{\ensuremath{{a}_{2,l}\left(t,x,\frac{\partial
			u_{2,\eps}}{\partial x_l}\right)}}
\newcommand{\ailtuepc}
{\ensuremath{{a}_{i,l}\left(t,x,\frac{\partial
			u_{i,\eps}}{\partial x_l}\right)}}
\newcommand{\ailtpa}
{\ensuremath{{a}_{1,l}\left(t,x,\frac{\partial u_1}{\partial
			x_l}\right)}}
\newcommand{\ailtpb}
{\ensuremath{{a}_{2,l}\left(t,x,\frac{\partial u_2}{\partial
			x_l}\right)}}
\newcommand{\ailtpc}
{\ensuremath{{a}_{i,l}\left(t,x,\frac{\partial v_i}{\partial
			x_l}\right)}}
\newcommand{\pxi}{\ensuremath{\frac{\partial \varphi_{\eps,j}}{\partial x_l}}}
\newcommand{\pxiet}{\ensuremath{\frac{\partial \eta_{\mu,j}(u_1)}{\partial x_l}}}
\newcommand{\pxitm}
{\ensuremath{\frac{\partial T_M(u_{1,\eps})}{\partial x_l}}}
\newcommand{\pxitg}{\ensuremath{\frac{\partial T_\gamma(u_{1,\eps})}{\partial x_l}}}
\newcommand{\pxiutg}{\ensuremath{\frac{\partial T_\gamma(u_1)}{\partial x_l}}}
\newcommand{\Sm}{\ensuremath{\sum_{l=1}^N}}

\newcommand{\ptuea}{\ensuremath{\frac{\partial u_{1,\eps}}{\partial x_l}}}
\newcommand{\ptueb}{\ensuremath{\frac{\partial u_{2,\eps}}{\partial x_l}}}
\newcommand{\ptuec}{\ensuremath{\frac{\partial u_{i,\eps}}{\partial x_l}}}
\newcommand{\ptua}{\ensuremath{\frac{\partial u_1}{\partial x_l}}}
\newcommand{\ptub}{\ensuremath{\frac{\partial u_2}{\partial x_l}}}
\newcommand{\ptuc}{\ensuremath{\frac{\partial v_i}{\partial x_l}}}

\newcommand{\pxlu}{\ensuremath{\pxl u}}
\newcommand{\pxlv}{\ensuremath{\pxl v}}
\newcommand{\pxlue}{\ensuremath{\pxl \ue}}
\newcommand{\sll}{\ensuremath{\sum_{l=1}^{N}}}

\newcommand{\Seque}{\ensuremath{\left(\ue\right)_{0<\eps\leq 1}}}
\newcommand{\elambda}{\ensuremath{e^{-\lambda t}}}

\newcommand{\dx}{\ensuremath{\, dx}}
\newcommand{\dz}{\ensuremath{\, dz}}
\newcommand{\dt}{\ensuremath{\, dt}}
\newcommand{\ds}{\ensuremath{\, ds}}
\newcommand{\dr}{\ensuremath{\, dr}}

\newcommand{\supp}{\ensuremath{\mathrm{supp}\,}}


\def\e{{\text{e}}}
\def\N{{I\!\!N}}

\numberwithin{equation}{section} \allowdisplaybreaks
	\newenvironment{proof}{\noindent{\it Proof.}}{\hfill$\square$}
	\let\WriteBookmarks\relax
	\def\floatpagepagefraction{1}
	\def\textpagefraction{.001}
	\shorttitle{Mathematical analysis of a stochastic reaction-diffusion system modeling predator-prey interactions with prey-taxis and noises}
	\shortauthors{M. Bendahmane, H. Nzeti, J. Tagoudjeu and M. Zagour}
	
	\title [mode = title]{Mathematical analysis of a stochastic reaction-diffusion system modeling predator-prey interactions with prey-taxis and noises}
	
	
	\author{Mostafa Bendahmane$^1$}
	\fnmark
	\ead{mostafa.bendahmane@u-bordeaux.fr}
	\address{$1$ Institut de Math\'ematiques de Bordeaux, Universit\'e de Bordeaux, 33076 Bordeaux Cedex, France}

	\author{Herbert Nzeti$^2$}
	\fnmark
	\ead{nzetiherbert@yahoo.fr}
	\address{$2$, $3$ École Nationale Supérieure Polytechnique de Yaoundé,
		Universite de Yaoundé I, B.P 8390 Yaoundé, Cameroun}
	\author{Jacques Tagoudjeu$^3$}
	\fnmark
	\ead{jacques.tagoudjeu@univ-yaounde1.cm}
	
	\author{Mohamed Zagour$^4$}
	\fnmark
	\ead{m.zagour@insa.ueuromed.org}
	\address{$4$ Euromed Research Center, Euromed University of Fes, Morocco}


	\begin{abstract}
	This paper is devoted to the mathematical analysis of a nonlinear stochastic reaction-diffusion system modeling predator-prey interactions with prey-taxis and noises. Precisely, we detail the proof of the existence of weak martingale solutions by Faedo-Galerkin approximations and the stochastic compactness method. We prove the nonnegativity of solutions by a stochastic adaptation of the Stampacchia approach. Finally, we prove the uniqueness of the solution via duality technique.

	\end{abstract}
	
	
\maketitle

\section{Introduction}

Population dynamics of prey-predator are one of the central themes of ecosystems to explain the evolution of organisms. The dynamic relationship between predators and their prey has been around for a long time as explained in  \cite{Berryman}. It is one of the dominant themes in ecology and mathematical ecology thanks to its universal existence and importance. Indeed, various mathematical models have been proposed to describe such a predator-prey relationship to predict long-term outcomes and impact on the whole ecosystem \cite{WWS2018}.  For instance, the pioneer Lotka-Volterra model is used to describe the dynamics of biological systems in which two prey and predator species interact \cite{Nint3}. The initial Lotka-Volterra model received many improvements, the most notable being the proper design of prey growth functions and the introduction of several functional responses ( see \cite{BCHN16, DS13} and their references).

Mathematical studies of the models of population dynamics have attracted many scientific interests and shown many essential features such as pattern formations that are commonly observed in natural ecological systems, more details can be found in \cite{YWS2009} and references therein. Moreover, it has been observed that several living species possess the ability to detect stimulating signals in the environment and therefore to adjust their movements. This phenomenon is known as taxis and has been studied by many authors, see for example \cite{BB5, DL21, JJJ21, OB12}. Mathematical models of a deterministic predator-prey system with prey-taxis have been proposed in \cite{ABN08,KO87}. Its different extensions have been studied in many works, see for instance \cite{DL20, JW17, RB21}.
In the case of predator-prey interactions, the mechanism of taxis is characterized by chase and flight, in which the predators move in the direction of the prey distribution gradient, called "prey-taxis", and/or the prey move opposite to the distribution of predators known as "predator-taxis", see \cite{WWS2018}. Thus, the prey-taxis describes the movement of predators towards the area with higher-density of prey population, playing a key role in biological control  and in ecological balance such as regulating prey population or incipient outbreaks of prey or forming large-scale aggregation for survival \cite{G1998,MCC1985,det23}. 

 As it is known, biological systems are subject to environmental fluctuations. Thus, the deterministic models have some limitations \cite{BC05, RW12}. Indeed, the explicit incorporation of stochasticity can fundamentally change and renormalize the behavior of the interacting species \cite{DMPT}.  Therefore, the basic mechanism and factors of population growth such as resources and vital rates-birth, and emigration-change non deterministically due to continuous fluctuations in the environment (e.g. variation in intensity of sunlight, water level) \cite{May01}. These fluctuations can be modeled by incorporating into the deterministic system  multiplicative noise sources which can effectively reproduce experimental data in population dynamic (see \cite{[BTZ22], DMPT, NY21} and the reference therein). Consequently, stochastic differential equations (SDEs) or stochastic partial differential equations (SPDEs) have attracted widespread scientific attention in population dynamics.

Several papers have investigated interesting mathematical properties of deterministic prey-predator models such as well-posedness, the positivity of solution, longtime dynamic behavior such as existence and uniqueness of stationary distribution, and optimal harvesting strategy, see \cite{HSWZ21, NY17, RA20, R03, Tao10, TY15, XY18}. In the case of the stochastic spatially dependent predator-prey models, without prey-taxis term, the authors in \cite{NY21, NY20} obtained the well-posedness and investigated the regularity of the solutions, the existence of density, the existence of an invariant measure for a stochastic reaction-diffusion system with non-Lipschitz and non-linear growth coefficients and multiplicative noise. Moreover, they have studied the existence and uniqueness, using the notion of a mild solution, and have derived sufficient conditions for persistence and extinction.

In this paper, we aim to study the mathematical analysis of the following nonlinear stochastic predator-prey system with prey-taxis:
\begin{equation}\label{S7}
	\left\{\begin{array}{rcl}
		\displaystyle d u_1 - d_1 \Delta u_1 \dt
		+\Div (\chi(u_1 )\nabla u_2)\dt=F_1(u_1,u_2)\dt
		+\sigma_{u_1}(u_1,u_2) \dW_{u_1}(t),
		{}\\
		\displaystyle du_2- d_2 \Delta u_2 \dt=F_2(u_1,u_2)\dt
		+\sigma_{u_2}(u_1,u_2) \dW_{u_2}(t),
	\end{array}\right.
\end{equation}
in $\Om_T$,
where $\Om_T:=\Omega\times (0,T)$, $T>0$ is a fixed time, and $\Omega$ is a bounded domain in $\mathbb{R}^N$ ($N=2$ or $3$), with smooth boundary $\partial \Omega$ and outer unit normal $\eta$.
In system \eqref{S7}, the functions $F_1$ and $F_2$ have the following form
\begin{equation}\label{reaction}
	\begin{split}
		&F_1(u_1,u_2)=e \pi(u_2)u_1-a u_1,\\
		&F_2(u_1,u_2)=k(u_2)-\pi(u_2)u_1.
	\end{split}
\end{equation}
The diffusion coefficients are denoted by $d_1$ and $d_2$. The coefficient $e$ is the conversion rate from prey to predator and $-a$ ($a>0$) be the natural exponential decay of the predator population. We consider the logistical growth rate of prey $k(u_2)=ru_2(1-\frac{u_2}{K})$, with $r > 0$ being the natural growth rate of prey and $K$ be the carrying capacity, and the predation rate $\pi(u_2)=p\,u_2/(1+q \,u_2)$ with $1/p$ the time spent by a predator to catch a prey and $q /p$ the manipulation time, offering a saturation effect for large densities of prey when $q > 0$. The predators are attracted by the prey and $\chi$ denotes their prey-tactic sensitivity. We assume that there exists a maximal density of their of predators, the threshold $u_m$, such that $\chi(u_m)=0$. This threshold condition can be interpreted as follows: the predators stop to accumulate at a given point of after their density attains certain threshold values while the prey-tactic cross-diffusion $\chi(u_1)$ vanishes identically whenever $u_1\geq u_m$. Therefore,
\begin{equation}\label{S3}
	\chi \in C^1(\mathbb{R}), \chi(u_1)=u_1(u_m-u_1)~~\textrm{if}~~0\leq u_1\leq u_m~~\textrm{and}~~\chi(u_1)=0~~\textrm{if no}.
\end{equation}
For our mathematical study we need to extend
the definitions of $F_1$ and $F_2$
to all $u_1,u_2 \in \R$. We do this by
assuming the following
\begin{equation}\label{entries-positive}
	\begin{split}
		F_1(u_1,u_2)=
		\begin{cases}
			e \pi(u_2)u_1-a u_1,
			& \text{if $u_1,u_2\ge 0$},\\
			-a u_1 ,
			& \text{if $u_1\ge 0$ and $u_2<0$},\\
			0, & \text{if $u_1< 0$ and $u_2\geq 0$ or $u_1,u_2<0$},
		\end{cases}
		\\
		F_2(u_1,u_2)=
		\begin{cases}
			k(u_2)-\pi(u_2)u_1,
			& \text{if $u_1,u_2\ge 0$},\\
			0, & \text{if $u_1\ge 0$ and $u_2<0$ or $u_1,u_2<0$},\\
			k(u_2),
			& \text{if $u_1< 0$ and $u_2\geq 0$}.
		\end{cases}	
	\end{split}
\end{equation}

In system \eqref{S7}, $W_{u_i}$ is a cylindrical Wiener process, with noise amplitude function $\sigma_{u_i}$ for $i= 1, 2$.
Formally one can consider $\sigma_{u_i}(u_1,u_2)\, dW_{u_i}$ as
$\sum_{k\ge1}\sigma_{{u_i},k}(u_1,u_2) \, d W_{k,{u_i}}(t)$, where $\{W_{k,{u_i}}\}_{k\ge 1}$
is a sequence of independent 1D Brownian motions and $\{\sigma_{{u_i},k}\}_{k\ge 1}$
a sequence of noise coefficients. Note that the noises $\dW_{u_1}$ and $\dW_{u_2}$ represent the independent environmental variables. Moreover, $\sigma_{u_1}(u_1,u_2) \dW_{u_1}$ and $\sigma_{u_2}(u_1,u_2) \dW_{u_2}$ model random perturbations of the stochastic predator-prey system with prey-taxis \eqref{S7}. 

\noindent We augment system (\ref {S7}) with no-flux boundary conditions on
$\Sigma_T:=\partial \Omega\times (0,T)$,
\begin{equation}\label{S5}
	\frac{\partial u_1}{\partial \eta}=0,\qquad \frac{\partial u_2}{\partial \eta}=0,
\end{equation}
and initial distributions in $\Omega$:
\begin{equation}\label{S6}
	u_1(x,0)=u_{1,0}(x),\quad u_2(x,0)=u_{2,0}(x).
\end{equation}

Let us now comment on the contribution of this paper. First, as the proposed system \eqref{S7} contains strong coupling in the highest derivative, the standard theory for stochastic parabolic systems can not apply naturally. Moreover, a stochastic forcing term complicates the maximum principle approach. The existence result for our system is based on martingale solutions and on the introduction of suitable approximate (Faedo-Galerkin) solutions. A series of system-specific a priori estimates are derived for the Faedo-Galerkin approximations and a compactness method to conclude convergence is used. In addition, as the structure of system \eqref{S7} is nonlinear, this requires strong convergence of the approximate solutions in suitable norms. We establish weak compactness of the probability laws of the approximate solutions, which follows from tightness and Prokhorov's theorem to deduce strong convergence in the probability variable. Then we construct almost sure (a.s.) convergent versions of the approximations using Skorokhod's representation theorem. We prove that the constructed solutions are nonnegative and uniformly bounded in $L^\infty$ according to the Stampacchia approach, see \cite{Chekroun:2016aa}. For the existence of martingale solutions for other classes of SPDEs, we refer the interested reader to \cite{[BK22],DaPrato:2014aa,Debussche:2011aa,H13,LW11,NY17,NY20,NY21}.  Finally, we prove the uniqueness of the solution via duality technique.

The paper is organized as follows:  In Section \ref{sec:stoch}, we present the stochastic framework and state the noise coefficients' hypotheses. Next, we supply the definition of a weak martingale solution and we declare our main result. Approximate solutions by the Faedo-Galerkin method is constructed in Section \ref{Sec3}. While, uniform estimates for these approximations are established in Sections \ref{Sec4}. Section \ref{Sec5} is devoted to ensure strong compactness of a sequence of Faedo-Galerkin solutions. Thus, we establish a temporal translation estimate in a space, which is enough to work out the required compactness (and tightness).
In Section \ref{Sec6}, we prove the tightness of the probability laws generated by the Faedo-Galerkin approximations. The tightness and Skorokhod's representation theorem is considered to show that a weakly convergent sequence of the probability laws has a limit that can be represented as the law of an almost surely convergent sequence of random variables defined on a common probability space. The limit of this sequence is proved to be a weak martingale solution of the stochastic system In Section \ref{Sec7}. Its nonnegativity and boundness in $L^\infty$ are deferred to Section \ref{Sec8} based on the Stampacchia method. Finally, the pathwise uniqueness result is established in Section \ref{sec:uniq}.

\section{Stochastic framework and notion of solution}\label{sec:stoch}
This section is devoted to recall some basic concepts and results from stochastic analysis (for more details see for instance \cite{barda1987optimal,prevot2007concise,karatzas1998brownian}). Next, we give the definition of a weak martingale solution to our stochastic predator-prey with prey-taxis system \eqref{S7}, \eqref{S5} and \eqref{S6}.
\subsection{Stochastic framework and notion of solution}
Let consider a complete probability space $(D,\cF, P)$, along
with a complete right-continuous filtration $\Set{\cFt}_{t\in [0,T]}$ (we assume that the $\sigma$-algebra
$\cF$ is countably generated). Equipped with the Borel $\sigma$-algebra $\cB(\bB)$, $\bB$ is a separable Banach
space. A $\bB$-valued random variable $X$ is a measurable mapping from $(D,\cF, P)$ to
$(\bB,\cB(\bB))$, $D\ni \omega\mapsto X(\omega)\in \bB$.
$\displaystyle\E[X]:=\int_{D} X\, dP$ is the expectation of a random variable $X$.\\
\noindent For $p\geq 1$, the Banach space $L^p(D,\cF,P)$ is the collection of all $\bB$-valued random variables, equipped with the following norm
\begin{align*}
	&\norm{X}_{L^p(D,\cF,P)}
	:=\left(\E\left[ \norm{X}_B^p\right]\right)^{\frac{1}{p}} \quad  (p<\infty),
	\\ &
	\norm{X}_{L^\infty(D,\cF,P)}:= \sup_{\omega\in D} \norm{X(\omega)}_B.
\end{align*}
We shall use the abbreviation
a.s.~(almost surely) for $P$-almost every $\omega\in D$.
A stochastic process $X=\Set{X(t)}_{t\in [0,T]}$ is a collection
of $\bB$-valued random variables $X(t)$. The stochastic process
$X$ is \textit{measurable} if the map $X:D\times [0,T]\to \bB$
is measurable from $\cF \times \cB([0,T])$ to $\cB(\bB)$.
The paths $t \to X (\omega,t)$ of a measurable
process $X$ are automatically Borel measurable functions.
A stochastic process $X$ is \textit{adapted} if $X(t)$ is $\cF_t$ measurable for
all $t\in [0,T]$. We refer to
\begin{equation}\label{eq:stochbasis}
	\cS=\left(D,\cF,\Set{\cFt}_{t\in [0,T]},P,\Set{W_k}_{k=1}^\infty\right)
\end{equation}
as a (Brownian) \textit{stochastic basis}, where $\Set{W_k}_{k=1}^\infty$
is a sequence of independent one-dimensional
Brownian motions adapted to the filtration $\Set{\cFt}_{t\in [0,T]}$.

Considering the Hilbert space $\U$ equipped with
a complete orthonormal basis $\Set{\psi_k}_{k\ge 1}$, we define the "cylindrical Brownian motions" $W$ on $\U$ by $W:=\sum_{k\ge 1} W_k \psi_k$.
The vector space of all bounded
linear operators from $\U$ to $\bX$ is denoted $L(\U,\bX)$, where $\bX$ is separable Hilbert space with inner product $(\cdot,\cdot)_{\bX}$ and
norm $\norm{\cdot}_{\bX}$. We denote by
$L_2(\U,\bX)$ the collection of Hilbert-Schmidt operators from $\U$ to $\bX$, that is to say,
$R\in L_2(\U,\bX)\Longleftrightarrow R\in L(\U,\bX)$ and
\begin{equation}\begin{split}\label{def:HS-norm}
		&\norm{R}_{L_2(\U,\bX)}:=\left(\sum_{k\geq 1}
		\norm{R\psi_k}_{\bX}^2\right)^{\frac12}<\infty\\
		&\left ( \hat R,\tilde R \right)_{L_2(\U,\bX)}
		=\sum_{k\ge 1} \left(\hat R \psi_k,\tilde R \psi_k\right)_X,
		\qquad \hat R, \tilde R\in L_2(\U,\bX).
\end{split}\end{equation}
Note that, for the stochastic predator-prey system with prey-taxis \eqref{S7}, a natural choice
is $\bX=L^2(\Om)$.
%
For a given a cylindrical Brownian motion $W_{u_i}$, we can define the It\^{o} stochastic
integral $\displaystyle\int \sigma_{u_i} \,dW_{u_i}$ as follows (see for e.g. \cite{DaPrato:2014aa,Prevot:2007aa}) for $i=1,2$
\begin{equation}\label{def:sint}
	\int_0^t \sigma_{u_i}\, dW_{u_i}=\sum_{k=1}^\infty \int_0^t \sigma_{{u_i},k} \,\dW_{{u_i},k},
	\qquad \sigma_{{u_i},k} := \sigma_{u_i} \psi_k,
\end{equation}
where $\sigma_{u_i}$ is a predictable $X$-valued
process satisfying
$$
\sigma_{u_i}\in L^2\Big(D,\cF,P;L^2((0,T);L_2(\U,\bX))\Big).
$$
The stochastic integral \eqref{def:sint} is an $\bX$-valued square integrable martingale,
satisfying the Burkholder-Davis-Gundy inequality
\begin{equation}\label{eq:bdg}
	\E\left[ \sup_{t\in [0,T]} \norm{\int_0^t \sigma_{u_i} \,dW_{u_i}}_{\bX}^p \right]
	\le C\, \E\left[\left(\int_0^T \norm{\sigma_{u_i}}_{L_2(\U,\bX)}^2 \dt\right)^{\frac{p}{2}} \right],
\end{equation}
for $i=1,2$, where $C>0$ is a constant depending on $p\ge 1$.

Note that since $W_{u_i}=\sum_{k\ge 1} W_{k,{u_i}} \psi_k$
is a cylindrical Brownian motion, we can give meaning to the following stochastic terms
\begin{equation}\label{eq:int-W}
	\begin{split}
		\int_{\Om} \left(\, \int_0^t  \sigma_{u_i}(u_1,u_2) \dW_{u_i}\right) \vphi \dx= \sum_{k\ge 1}\int_0^t \int_{\Om} \sigma_{{u_i},k}(u_1,u_2) \vphi \dx \dW_{{u_i},k}
		\qquad \text{for }i= 1, 2,
	\end{split}
\end{equation}
where $ \vphi\in L^2(\Om)$ and $\sigma_{{u_i},k}(u_1,u_2):=\sigma_{u_i}(u_1,u_2)\psi_k$ are real-valued functions.

We impose conditions on the noise
$\sigma_{u_i}$. For each ${u_i}\in L^2(\Om)$,
we assume that $\sigma_{u_i}(u_1,u_2):\U\to L^2(\Om)$ is defined by
$$
\sigma_{u_i}(u_1,u_2)\psi_k=\sigma_{{u_i},k}(u_1(\cdot),u_2(\cdot)),  \quad k\ge 1,\qquad \text{for }i= 1, 2,
$$
for some real-valued functions $\sigma_{{u_i},k }(\cdot,\cdot):\R^2\to \R$ that satisfy (for $i= 1, 2$)
\begin{equation}\label{eq:noise-cond}
	\begin{split}
		& \sum_{k\ge 1} \abs{\sigma_{{u_i},k }(u_1,u_2)}^2 \le
		C_\sigma \left(1+ \abs{u_1}^2+\abs{u_2}^2\right), \qquad \forall u_1,u_2\in \R,
		\\
		&\sum_{k\ge 1} \abs{\sigma_{u_i,k }(\bar u_1,\bar u_2)-\sigma_{u_ik }(\hat u_1,\hat u_2)}^2
		\le C_\sigma\Big(\abs{\bar u_1-\hat u_1}^2+\abs{\bar u_2-\hat u_2}^2\Big), \qquad \forall \bar u_1,\bar u_2, \hat u_1,\hat u_2\in \R,
	\end{split}
\end{equation}
for a constant $C_\sigma>0$. Consequently,
\begin{equation}\label{eq:noise-cond2}
	\begin{split}
		&\norm{\sigma_{u_i}(u_1,u_2)}_{L_2\left(\U,L^2(\Om)\right)}^2
		\le C_\sigma \left(1 + \norm{u_1}_{L^2(\Om)}^2+\norm{u_2}_{L^2(\Om)}^2\right),
		\quad \forall u_1,u_2\in L^2(\Om),
		\\ &
		\norm{\sigma_{u_i}(\bar u_1,\bar u_2)-\sigma_{u_i}(\hat u_1,\hat u_2)}_{L_2\left(\U,L^2(\Om)\right)}^2
		\le C_\sigma \Big(\norm{\bar u_1-\hat u_1}_{L^2(\Om)}^2+\norm{\bar u_2-\hat u_2}_{L^2(\Om)}^2\Big), \quad
		\forall \bar u_1,\bar u_2, \hat u_1,\hat u_2 \in L^2(\Om),
	\end{split}
\end{equation}
for $i= 1, 2$.

We denote by $\cB(\bA)$ the family
of the Borel subsets of $\bA$ and by $\cP(\bA)$ the family of all
Borel probability measures on $\bA$, where $\bA$ is a separable Banach (or Polish) space.
Note that, each random variable $X:D\to \bA$ induces
a probability measure on $\bA$ via the pushforward $X_\# P:=P\circ X^{-1}$.
Finally, a sequence of probability measures $\Set{\mu_n}_{n\ge1}$ on $(\bA, \cB(\bA))$ is
tight if for every $\epsilon>0$ there is a compact
set $\bK_\epsilon\subset \bA$ such that $\mu_n(\bK_\epsilon)>1-\epsilon$ for all $n\ge 1$.

\subsection{Notion of solution and existence results}\label{sec:defsol}

We start by giving the definition of a weak martingale solution. Next, we state our existence results.

\begin{defi}[Weak martingale solution] \label{def:martingale-sol}
	Let $\mu_{u_{1,0}}$ and $\mu_{u_{2,0}}$ be probability measures on $L^2(\Om)$.
	A weak martingale solution of the stochastic predator-prey-taxis system \eqref{S7}, \eqref{S5} and \eqref{S6}, is a collection $\bigl(\cS,u_1,u_2\bigr)$
	satisfying
	\begin{enumerate}
		\item\label{eq:mart-sto-basis}
		$\cS=\left(D,\cF,\Set{\cFt}_{t\in [0,T]},P,\Set{W_{k,u_1}}_{k=1}^\infty,
		\Set{W_{k,u_2}}_{k=1}^\infty\right)$ is a stochastic basis;
		
		\item \label{eq:mart-wiener}
		$W_{u_1}:=\sum_{k\ge 1} W_{k,u_1} \psi_k$ and
		$W_{u_2}:=\sum_{k\ge 1} W_{k,u_2} \psi_k$ are
		two independent  cylindrical Brownian motions, adapted
		to the filtration $\Set{\cFt}_{t\in [0,T]}$;
		
		\item\label{eq:mart-uiue-reg}
		For $P$-a.e.~$\omega\in D$, $u_1(\omega),u_2(\omega)$ are nonnegative and\\
		$u_1(\omega),u_2(\omega)
		\in L^\infty\big((0,T);L^2(\Om)\big)\cap L^\infty(\Om_T)\cap L^2\big((0,T);\tH(\Om)\big)$.
		
		\item\label{eq:mart-data}
		The laws of $u_{1,0}:=u_{1}(0)$ and $u_{2,0}:=u_{2}(0)$ are
		respectively $\mu_{u_{1,0}}$ and $\mu_{u_{2,0}}$:
		$$
		P\circ u_{1,0}^{-1}=\mu_{u_{1,0}}, \qquad
		P\circ u_{2,0}^{-1}=\mu_{u_{2,0}};
		$$
		
		\item\label{eq:mart-weakform}
		The following identities hold $P$-almost surely, for any $t \in [0,T]$
		\begin{equation}\label{eq:weakform}
			\begin{split}
				& \int_{\Om} u_1(t) \vphi_{u_1} \dx
				+ d_1\int_0^t\int_{\Om} \Grad u_1 \cdot \Grad \vphi_{u_1}  \dx\ds - \int_0^t\int_{\Om} \chi(u_1)  \Grad u_2 \cdot \Grad \vphi_{u_1}  \dx\ds
				\\ &
				= \int_{\Om} u_{1,0} \, \vphi_{u_1} \dx +\int_0^t \int_{\Om}F_1(u_1,u_2)\vphi_{u_1}  \dx\ds
				+  \int_0^t \int_{\Om} \sigma_{u_1}(u_1,u_2) \vphi_{u_1} \dx  \dW_{u_1} (s),
				\\ &
				\int_{\Om} u_2(t)  \vphi_{u_2}\dx+d_2\int_0^t\int_{\Om} \Grad u_2\cdot \Grad \vphi_{u_2} \dx\ds
				\\ &
				=\int_{\Om} u_{2,0} \vphi_{u_2} \dx+\int_0^t \int_{\Om} F_2(u_1,u_2) \vphi_{u_2}\dx\ds
				+ \int_0^t \int_{\Om} \sigma_{u_2}(u_1,u_2)  \vphi_{u_2} \dx \dW_{u_2}(s),
			\end{split}
		\end{equation}
		\noindent for all $\vphi_{u_1},\vphi_{u_2} \in H^{1}(\Omega)$.
	\end{enumerate}
\end{defi}

Our main result is the following existence and uniqueness theorem for weak solutions.
\begin{thm}[Existence of weak martingale solution]\label{thm}
	Assume \eqref{S3} and \eqref{eq:noise-cond} hold and the initial condition $(u_{1,0},u_{2,0})$ is nonnegative
and bounded in $L^\infty$. Let $\mu_{u_{1,0}}$, $\mu_{u_{2,0}}$
	be probability measures satisfying
	\begin{equation}\label{cond-init}
		\int_{L^2(\Om)} \norm{u_i}^r_{L^2(\Om)} d \mu_{u_{ i,0}}(u_i)<+\infty \qquad  \text{ for $i= 1, 2$ and $r>2$}.
	\end{equation}
	Then the stochastic predator-prey-taxis system \eqref{S7}, \eqref{S5} and \eqref{S6}
	possesses a unique weak martingale solution in the
	sense of Definition \ref{def:martingale-sol}.
\end{thm}

\section{Construction of stochastic Faedo-Galerkin solutions}
\label{Sec3}
This section is devoted to define precisely the Faedo-Galerkin
equations and prove that there exists
a solution to these equations. We start
by fixing a stochastic basis $\cS$,
cf.~\eqref{eq:stochbasis}, and $\cF_0$-measurable
initial data $u_{1,0},u_{2,0} \in L^2(D;L^2(\Om))$,
with respective laws $\mu_{u_{1,0}},\mu_{u_{2,0}}$ on $L^2(\Om)$.
We are looking for approximate solutions obtained from the
projection of \eqref{S7}, \eqref{S5} and \eqref{S6} onto a finite dimensional space
$\bX_n:=\Span\Set{e_1,\ldots,e_n}$, where the sequence
$\Set{e_\ell}_{\ell=1}^\infty$ is an orthonormal basis of $L^2(\Om)$. The $L^2$ orthogonal projection is denoted by
\begin{equation}\label{eq:L2-project}
	\Pi_n: L^2(\Om)\to \bX_n=\Span\Set{e_1,\ldots,e_n},
	\quad \Pi_n u := \sum_{\ell=1}^n
	\left( u,e_\ell\right)e_\ell.
\end{equation}
We consider the following approximations of the noise coefficients:
\begin{equation}\begin{split}\label{def-sigwkl}
		&\sigma_{u_i,k}^n(u_1^n,u_2^n)
		:=\sum_{\ell=1}^n \sigma_{u_i,k,\ell}(u_1^n,u_2^n) e_\ell,
		\quad
		\text{where}
		\\ &
		\sigma_{u_i,k,\ell}(u_1^n,u_2^n)
		:=\left(\sigma_{u_i,k}(u_1^n,u_2^n),e_\ell\right)_{L^2(\Om)},
		\quad i=1,2.
\end{split}\end{equation}
Now, let define our Faedo-Galerkin approximations
\begin{equation}\label{eq:approxsol}
	\begin{split}
		&u_1^n,u_2^n:[0,T]\to \bX_n,
		\quad
		u_{1}^n(t)=\sum_{\ell=1}^n c_{1,\ell}^n(t)e_\ell,
		\quad u_2^n(t)=\sum_{\ell=1}^n c_{2,\ell}^n (t)e_\ell,
	\end{split}
\end{equation}
where the coefficients $c_1^n=\Set{c_{1,\ell}^n(t)}_{\ell=1}^n$
and $c_2^n=\Set{c_{2,\ell}^n}_{\ell=1}^n$ are
determined such that the following
equations hold (for $\ell=1,\ldots, n$):
\begin{equation}\label{S1-Galerkin-new}
	\begin{split}
		& \left(d u_1^n, e_\ell\right)
		+ d_1 \left(\Grad u_1^n,\Grad e_\ell\right)\dt
		-\bigl ( \chi(u_1^n) \Grad u_2^n , \Grad e_\ell \bigr)\dt
		\\ & \qquad \qquad
		= \left (F_1(u_1^n,u_2^n), e_\ell \right)\dt
		+ \sum_{k=1}^n\left( \sigma^n_{u_1,k}(u_1^{n},u_2^n),
		e_\ell \right)\dW_{u_1,k}(t),
		\\ &
		\left(d u_2^n, e_\ell\right)
		+d_2\left(\Grad u_2^n, \Grad e_\ell\right) \dt
		\\ & \qquad \qquad
		= \left(F_2(u_1^n,u_2^n),e_\ell\right)\dt
		+\sum_{k=1}^n\left(\sigma^n_{u_2,k}(u_1^{n},u_2^n),
		e_\ell \right)\dW_{u_2,k}(t),
	\end{split}	
\end{equation}
and, with reference to the initial data,
\begin{equation}\label{Galerkin-initdata}
	\begin{split}
		& u_1^n(0)=u_{1,0}^n :=\sum_{\ell=1}^n c_{1,\ell}^n(0)e_\ell,
		\quad
		c_{1,\ell}^n(0):=\left(u_{1,0}^n,e_\ell\right)_{L^2(\Om)},
		\\ &
		u_2^n(0)=u_{2,0}^n :=\sum_{\ell=1}^n c_{2,\ell}^n(0) e_\ell,
		\quad
		c_{2,\ell}^n(0):=\left(u^n_{2,0},e_\ell\right)_{L^2(\Om)}.
	\end{split}
\end{equation}
Using the basic properties of the projection
operator $\Pi_n$, we obtain
\begin{equation}\label{eq:approx-eqn-integrated}
	\begin{split}
		&u_1^n(t)-u_1^n(0)
		-\int_0^t \Pi_n \left[\Div \left(d_1 \Grad u_1^n - \chi(u_1^n) \Grad u_2^n \right) \right]\ds
		\\ & \qquad\qquad
		=\int_0^t \Pi_n\left[F_1(u_1^n,u_2^n)\right]\ds+\int_0^t\sigma^n_{u_1}(u_1^{n},u_2^n) \dW^n_{u_1}(s)
		\quad \text{in $\left(H^1(\Om)\right)^\star$},
		\\ & u_2^n(t)-u_2^n(0)-d_2 \int_0^t \Pi_n \left[\Delta u_2^n \right]\ds
		\\ & \qquad\qquad =\int_0^t \Pi_n\left[F_2(u_1^n,u_2^n)\right]\ds
		+\int_0^t \sigma^n_{u_2}(u_1^n,u_2^{n})\dW^n_{u_2}(s)
		\quad \text{in $\left(H^1(\Om)\right)^\star$},
	\end{split}
\end{equation}
with initial data $u_{1,0}^n=\Pi_n u_{1,0}$
and $u_{2,0}^n=\Pi_n u_{2,0}$. Observe that
System \eqref{eq:approx-eqn-integrated}
allows to treat $u_1^n$, $u_2^n$ as stochastic
processes in $\R^n$, therefore we can apply the
finite dimensional It\^{o} formula to the Faedo-Galerkin equations.

The existence of pathwise solutions to the
finite-dimensional problem
\eqref{S1-Galerkin-new}, \eqref{Galerkin-initdata}
is given in the following lemma.

\begin{lem}\label{lem:fg-solutions}
	For each $n\in \N$, the Faedo-Galerkin equations
	\eqref{eq:approxsol}, \eqref{S1-Galerkin-new},
	\eqref{Galerkin-initdata} possess a unique adapted
	solution $(u_1^n(t), u_2^n(t))$ on $[0,T]$.
	Moreover, $u_1^n,u_2^n\in C([0,T];\bX_n)$ a.s.,
	where $\E\bigl[\norm{u_i^n(t)}_{L^2(\Om)}^2\bigr]
	\lesssim_{T,n} 1$, $\forall t\in[0,T]$, $i=1,2$.
\end{lem}

\begin{proof}
	We are looking for a stochastic process $C^n$ taking values
	in $\bX_n\times \bX_n$ solution to the following system of
	stochastic differential equations
	\begin{equation}\label{S1-Galerkin-bis}
		\begin{split}
			d C^n= M(C^n)\dt
			+\Gamma(C^n)\dW^{n},
		\end{split}
	\end{equation}
	where $C^n=\begin{pmatrix}
		u_1^n \\
		u_2^n
	\end{pmatrix}$,
	$M(C^n)=
	\begin{pmatrix}
		A_{u_1}\left(C^n\right)\\
		A_{u_2}\left(C^n\right)
	\end{pmatrix}$,
	\begin{align*}
		& A_{u_1}\left(C^n\right)
		= -\Pi_n\Div\Big(d_1 \Grad u_1^n - \chi(u_1^n) \Grad u_2^n \Big)
		+\Pi_n F_1(u_1^n,u_2^n),
		\\
		& A_{u_2}\left(C^n\right)
		= -\Pi_n\Div  \left(d_2 \Grad u_2^n \right)
		+\Pi_n F_2(u_1^n,u_2^n).
	\end{align*}
	and
	$$\Gamma(C^n)\dW^n:=
	\begin{pmatrix}
		\sigma^n_{u_1}\left(u_1^n,u_2^n\right)\dW_{u_1}^n\\
		\sigma^n_{u_2}\left(u_1^n,u_2^n\right)\dW_{u_2}^{n}
	\end{pmatrix}
	.
	$$
	We complete system \eqref{S1-Galerkin-bis} with
	initial data $C^n(0)=C_0^n$, where $C_0^n$
	is the vector given by \eqref{Galerkin-initdata}.
	\noindent Exploiting the global Lipschitz continuity of $F_1,F_2,\Gamma$, we deduce easily
	the weak coercivity condition:
	for all
	$C=\begin{pmatrix}
		u_1 \\
		u_2
	\end{pmatrix}
	\in \bX_n \times \bX_n$,
	\begin{equation}\label{weak-coercivity-bis}
		2\bigl(M(C),C\bigr)
		+\norm{\Gamma(C)}_{L^2(\Om)}^2
		\leq K \left(1+\norm{C}_{L^2(\Om)}^2\right),
	\end{equation}
	for some constant $K>0$. Next step is to prove the following local weak monotonicity:
	for all
	$C_1=\begin{pmatrix}
		u_1 \\
		u_2
	\end{pmatrix}\in \bX_n \times \bX_n$ and
	$C_2=\begin{pmatrix}
		\tilde u_1 \\
		\tilde u_2
	\end{pmatrix}\in \bX_n \times \bX_n$ such that
	$\norm{u_i^n}_{L^2(\Om)},
	\norm{\tilde u_i^n}_{L^2(\Om)}\leq r$,
	for any $r>0$ and $i=1,2$, we have
	\begin{equation}\label{weak-monotonicity}
		\begin{split}
			& 2 \bigl( M(C_1)-M(C_2),C_1-C_2\bigr)
			+\norm{\Gamma(C_1)-\Gamma(C_2)}_{L^2(\Om)}^2
			\\ & \qquad
			\leq K(r)\norm{C_1-C_2}_{L^2(\Om)}^2,
		\end{split}
	\end{equation}
	for a constant $K(r)$ that may depend on $r$, where
	$(\cdot,\cdot)$ denotes the $L^2(\Om)$ inner product. To do this we fix a real number $r>0$ and we set
	$U_1:=u_1-\tilde u_1$ and
	$U_2 :=u_2-\tilde u_2$, where $u_i,\tilde u_i$
	are arbitrary functions in $\bX_n$ for which
	$\norm{u_i}_{L^2(\Om)},\norm{\tilde u_i}_{L^2(\Om)}\leq r$
	for $i=1,2$. Thanks to Young's inequality, we have the following equality
	\begin{equation}\label{eq:mon-check}
		\bigl( M(C_1)-M(C_2),C_1-C_2\bigr)
		+\norm{\Gamma(C_1)-\Gamma(C_2)}_{L^2(\Om)}^2
		= \sum_{i=0}^5 I_i,
	\end{equation}
	where $I_0=\norm{\Gamma(C_1)
		-\Gamma(C_2)}_{L^2(\Om)}^2
	\overset{\eqref{eq:noise-cond2}}{\lesssim}
	\norm{C_1-C_2}_{L^2(\Om)}^2$ and
	\begin{align*}
		&I_1 = - \sum_{i=1,2}
		d_i\bigl( \Grad U_i,
		\Grad U_i\bigr) \leq 0,
		\\ & I_2 =
		\left(
		\begin{pmatrix}
			\chi(u_1) \Grad U_1
			\\
			0
		\end{pmatrix},
		\begin{pmatrix}
			\Grad U_1
			\\
			\Grad U_2
		\end{pmatrix}\right),
		\\ & I_3= \left(
		\begin{pmatrix}
			\Bigl(\chi(u_1)-\chi(\tilde u_1)\Bigr)  \Grad \tilde u_2
			\\
			0
		\end{pmatrix},
		\begin{pmatrix}
			\Grad U_1
			\\
			\Grad U_2
		\end{pmatrix}
		\right),
		\\ & I_4=
		\bigl(F_1(u_1,u_2)-F_1(\tilde u_1,\tilde u_2),
		U_1\bigr),
		\quad I_5=
		\bigl (F_2(u_1,u_2)-F_2(\tilde u_1,\tilde u_2),
		U_2\bigr).
	\end{align*}
	According to \eqref{S3} and H\"older inequality, we obtain
	\begin{align*}
		\abs{I_3}
		& \lesssim
		\norm{u_1-\tilde u_1}_{L^2(\Om)}
		\norm{\Grad \tilde u_2}_{L^4(\Om)}
		\norm{\Grad U_1}_{L^4(\Om)}
		\\ & \lesssim \norm{u_1-\tilde u_1}_{L^2(\Om)}
		\norm{\Grad \tilde u_2}_{H^1(\Om)}
		\norm{\Grad U_1}_{H^1(\Om)},
	\end{align*}
	thus $\displaystyle\abs{I_3}	\lesssim_{r,n}
	\sum_{i=1,2}\norm{u_i-\tilde u_i}_{L^2(\Om)}$.
	On the basis of the global Lipschitz continuity
	of the reaction functions $F_1$ and $F_2$,
	cf.~\eqref{reaction}, we have the following estimate
	$$
	\abs{I_4}+\abs{I_5}\lesssim
	\sum_{i=1,2}\norm{u_i-\tilde u_i}_{L^2(\Om)}
	\sum_{i=1,2}\norm{U_i}_{L^2(\Om)},
	$$
	thus $\abs{I_4}+\abs{I_5}\lesssim_{r}
	\sum_{i=1,2}\norm{u_i-\tilde u_i}_{L^2(\Om)}$.
	According to \eqref{eq:mon-check}, we obtain
	$\sum_{i=0}^5 I_i \lesssim_{r,n}
	\norm{C_1^n-C_2^n}_{L^2(\Om)}^2$, and
	\eqref{weak-monotonicity} is achieved.
	Finally the existence and uniqueness of a pathwise
	solution to \eqref{S1-Galerkin-bis} is a consequence of \eqref{weak-coercivity-bis}
	and \eqref{weak-monotonicity} (see for more details, \cite[Theorem 3.1.1]{Prevot:2007aa}).

\end{proof}

\section{Basic a priori estimates}\label{Sec4}

This section provides a series of basic
energy-type estimates.

\begin{lem}\label{lem:apriori-est}
	Let $u_1^n(t),u_2^n(t)$, $t\in [0,T]$,
	satisfy \eqref{S1-Galerkin-new}, \eqref{Galerkin-initdata}.
	There is a constant $C>0$, independent
	of $n$, such that
	\begin{align}
		& \E \left[ \norm{u_1^n(t)}_{L^2(\Om)}^2\right]
		+\E \left[ \norm{u_2^n(t)}_{L^2(\Om)}^2\right]
		\le C, \qquad \forall t\in [0,T];
		\label{Gal:est1}
		\\ &
		\E \left[\int_0^T \int_{\Om}\abs{\Grad u_1^n}^2
		\dx \dt \right]
		+\E \left[\int_0^T \int_{\Om} \abs{\Grad u_2^n}^2
		\dx\dt \right] \le C;
		\label{Gal:est2}
		\\ & \E \left[ \sup_{t\in [0,T]}
		\norm{u_1^n(t)}_{L^2(\Om)}^2\right]
		+\E \left[ \sup_{t\in [0,T]}
		\norm{u_2^n(t)}_{L^2(\Om)}^2\right]\le C.
		\label{Gal:est3}
	\end{align}
\end{lem}

\begin{proof}
	
	According to It\^{o}'s formula,
	$dS(u_i^n)\!=\!S'(u_i^n)\, du_i^n + \frac12 S''(u_i^n)\sum_{k=1}^n
	\left(\sigma_{u_i,k}(u_i^n)\right)^2\dt$, $i=1,2$,
	for any $C^2$ function $S:\R\to \R$.
	With $S(u_i)=\frac12 \abs{u_i}^2$ for $i=1,2$, we get
	\begin{equation}\label{eq:v+w-1}
		\begin{split}
			&\frac12 \sum_{i=1,2}\norm{u_i^n(t)}_{L^2(\Om)}^2
			+\sum_{i=1,2} d_i \int_0^t
			\int_{\Om} \abs{\Grad u_i^n}^2\dx\ds
			\\ & \quad =
			\frac12 \sum_{i=1,2}\norm{u_i^n(0)}_{L^2(\Om)}^2
			+\sum_{i=1,2}\int_0^t\Bigl(F_i(u_1^n,u_2^n),u_1^n\Bigr)_{L^2(\Om)}\ds
			\\ & \qquad
			+\sum_{i=1,2}\sum^n_{k=1}\int_0^t \int_{\Om}
			u_i^n \sigma_{u_i,k}^n(u_1^n,u_2^n)\dx \dW_{u_i,k}
			+\frac12
			\sum_{i=1,2} \sum^n_{k=1}
			\int_0^t \int_{\Om}
			\left(\sigma_{u_i,k}^n(u_1^n,u_2^n)\right)^2 \dx\ds
			\\ & \qquad
			+\int_0^t\Bigl(\chi(u_1^n) \nabla u_2^n,\Grad u_1^n\Bigr)_{L^2(\Om)}\ds
			\\ & \quad
			\leq \frac12 \sum_{i=1,2}\norm{u_i^n(0)}_{L^2(\Om)}^2
			+C \int_0^t \left(1+\norm{u_1^n(t)}_{L^2(\Om)}^2
			+\norm{u_2^n(t)}_{L^2(\Om)}^2\right)\ds
			\\ & \qquad
			+\sum_{i=1,2} \sum^n_{k=1}\int_0^t \int_{\Om}
			u_i^n\sigma_{u_i,k}^n(u_1^n,u_2^n)\dx\dW_{u_i,k} (s)
                            +\frac12
			\sum_{i=1,2} \sum^n_{k=1}
			\int_0^t \int_{\Om}
			\left(\sigma_{u_i,k}^n(u_1^n,u_2^n)\right)^2 \dx\ds
			\\
			& \qquad
			+\frac{d_1}{2}\int_0^t \int_{\Om} \abs{\Grad u_1^n}^2\dx\ds
			+C(d_1,u_m)\int_0^t \int_{\Om} \abs{\Grad u_2^n}^2\dx\ds\\
			&\quad \leq \frac12 \sum_{i=1,2}\norm{u_i^n(0)}_{L^2(\Om)}^2
			+C \int_0^t \left(1+\norm{u_1^n(t)}_{L^2(\Om)}^2
			+\norm{u_2^n(t)}_{L^2(\Om)}^2\right)\ds
			\\ & \qquad
			+\sum_{i=1,2} \sum^n_{k=1}\int_0^t \int_{\Om}
			u_i^n\sigma_{u_i,k}^n(u_1^n,u_2^n)\dx\dW_{u_i,k} (s)
   			+\frac12 \sum_{i=1,2} \sum^n_{k=1} \int_0^t \int_{\Om}
			\left(\sigma_{u_i,k}^n(u_1^n,u_2^n)\right)^2 \dx\ds
			\\
			& \qquad +\frac{d_1}{2}\int_0^t \int_{\Om} \abs{\Grad u_1^n}^2\dx\ds
			+\frac{C(d_1,u_m)}{d_2}
			\Biggl(\frac12\norm{u_2^n(0)}_{L^2(\Om)}^2
			+C \int_0^t \left(1+\norm{u_1^n(t)}_{L^2(\Om)}^2
			+\norm{u_2^n(t)}_{L^2(\Om)}^2\right)\ds
			\\& \qquad \qquad  \qquad  \qquad  \qquad
			+\sum^n_{k=1}\int_0^t \int_{\Om}
			u_2^n\sigma_{u_2,k}^n(u_1^n,u_2^n)\dx\dW_{u_2,k} (s)
                           +\frac12 \sum^n_{k=1} \int_0^t \int_{\Om}
			\left(\sigma_{u_2,k}^n(u_1^n,u_2^n)\right)^2 \dx\ds\Biggl),
		\end{split}
	\end{equation}
	where we have used the global Lipschitz of the reaction functions in \eqref{reaction} and Young inequality.
	Using \eqref{eq:noise-cond2}, \eqref{eq:v+w-1} implies
	\begin{equation}\label{eq:v+w-new}
		\begin{split}
			&\sum_{i=1,2} \norm{u_i^n(t)}_{L^2(\Om)}^2
			+ \sum_{i=1,2}
			\frac{d_1}{2}\int_0^t \int_{\Om}\abs{\Grad u_1^n}^2\dx\ds
			+d_2\int_0^t \int_{\Om}\abs{\Grad u_2^n}^2\dx\ds
			\\ & \quad
			\leq \sum_{i=1,2}\norm{u_i^n(0)}_{L^2(\Om)}^2		+C\int_0^t \left(1+
			\sum_{i=1,2}\norm{u_i^n(t)}_{L^2(\Om)}^2\right)\ds
			+C \sum_{i=1,2} \sum^n_{k=1}\int_0^t \int_{\Om}
			u_i^n\sigma_{u_i,k}^n(u_1^n,u_2^n)\dx\dW_{u_i,k} (s).
		\end{split}
	\end{equation}

	Now we apply $\E[\cdot]$ to \eqref{eq:v+w-new}, we
         exploit that the initial data
	$u_{1,0},u_{2,0}$ belong to $L^2$ a.s.,
	$$
	\displaystyle \E\Biggl[ \sum^n_{k=1}\int_0^t \int_{\Om}
	u_i^n\sigma_{u_i,k}^n(u_1^n,u_2^n)\dx\dW_{u_i,k} (s)\Biggl]=0,
	$$
	for $i=1,2$,
	and we use the Gronwall inequality, to arrive
	at \eqref{Gal:est1} and \eqref{Gal:est2}. \\
	
	\noindent To prove estimate \eqref{Gal:est3}, we
	take $\sup_{t\in [0,T]}$ and then $\E[\cdot]$
	in \eqref{eq:v+w-1} and \eqref{eq:v+w-new}. Using \eqref{Gal:est1}
	and the $L^2$ boundedness
	of the initial data, we end up
	with the estimate
	\begin{equation}\label{Esup1}
		\sum_{i=1,2}\E\left[ \sup_{t\in [0,T]}
		\norm{u_i^n(t)}_{L^2(\Om)}^2\right]
		\le C\left(1 + \sum_{i=1,2} I_{u_i}\right),
	\end{equation}
	where
	$\displaystyle I_{u_i}:=\E\left[\, \sup_{t\in [0,T]}
	\abs{\sum_{k=1}^n\int_0^t\int_{\Om}
		u_i^n \sigma_{u_i,k}^n(u_1^n,u_2^n) \dx \dW_{u_i,k} (s)}
	\, \right]$.
	Using the BDG inequality \eqref{eq:bdg},
	the Cauchy-Schwarz inequality, \eqref{eq:noise-cond}, Cauchy's
	inequality, and \eqref{Gal:est1}, we proceed as
	follows for $i=1,2$:
	\begin{equation}\label{Esup2}
		\begin{split}
			\abs{I_{u_i}} & \le
			C \E \left[ \left(\int_0^T \sum_{k=1}^n
			\abs{\int_{\Om} u_i^n \sigma^n_{u_i,k}(u_1^n,u_2^n) \dx}^2
			\dt \right)^{\frac12}\right]
			\\ &  \le
			C \E \left[ \left(\int_0^T
			\left(\int_{\Om} \abs{u_i^n}^2\dx\right)
			\left(\sum_{k=1}^n\int_{\Om}
			\abs{\sigma^n_{u_i,k}(u_1^n,u_2^n)}^2
			\dx \right)\dt\right)^{\frac12} \right]
			\\ &
			\le C \E \left[ \left(\sup_{t\in [0,T]}\int_{\Om}
			\abs{u_i^n}^2\dx\right)^{\frac12}
			\left( \int_0^T\sum_{k=1}^n\int_{\Om}
			\abs{\sigma^n_{u_i,k}(u_1^n,u_2^n)}^2
			\dx\dt \right)^{\frac12}\right]
			\\ &
			\le \frac12 \E \left[\sup_{t\in [0,T]}
			\int_{\Om} \abs{u_i^n}^2\dx\right]
			+C\E\left[ \int_0^T\sum_{k=1}^n
			\int_{\Om} \abs{\sigma^n_{u_i,k}(u_1^n,u_2^n)}^2
			\dx\dt\right]
			\\ &
			\le \frac12 \E\left[ \sup_{t\in [0,T]}
			\norm{u_i^n(t)}_{L^2(\Om)}^2\right]+\tilde{C},
		\end{split}
	\end{equation}
	for some sonstants $C,\tilde{C}>0$. Combining the
	inequalities \eqref{Esup1} and \eqref{Esup2}, we arrive at
	the estimate \eqref{Gal:est3}.
\end{proof}

Now, let consider $u_{1,0},u_{2,0} \in L^{q}\left(D,\cF,P;L^2(\Om))\right)$
with $q\in (2,q_0]$ and $q_0>3$. Using \eqref{eq:v+w-1},
the following estimate holds
for any $(u_i,t)\in D \times [0,T]$:
\begin{equation*}
	\begin{split}
		&\sum_{i=1,2}\sup_{0\le \tau\le t}
		\norm{u_i^n(\tau)}_{L^2(\Om)}^2+\sum_{i=1,2} d_i\int_0^t \norm{\Grad u_i(s)}_{L^2(\Om)}^2 \ds
		\le \sum_{i=1,2}\norm{u_i^n(0)}_{L^2(\Om)}^2
		+C\sum_{i=1,2}\int_0^t \norm{u_i^n(s)}_{L^2(\Om)}^2 \ds
		\\ & \qquad\qquad\qquad\qquad\qquad\qquad
		+C\sum_{i=1,2}\sup_{0\le \tau\le t}
		\abs{\sum_{k=1}^n \int_0^{\tau}\int_{\Om}
			u_i^n \sigma_{u_i,k}^n(u_1^n,u_2^n) \dx \dW_{u_i,k} (s)},
	\end{split}
\end{equation*}
We raise both sides of this inequality to power
$q/2$ and we take the expectation. Consequently,
\begin{equation}\label{eq:Lq0-estimation}
	\E\left[ \, \sup_{0\le t\le T}
	\norm{u_i^n(t)}_{L^2(\Om)}^{q}\, \right]\le C,
	\quad
	\E\left[
	\norm{\Grad u_i^n}_{L^2((0,T)\times \Om)}^{q}
	\right]\le C,
	\quad i=1,2.
\end{equation}
for some constant $C>0$, independent of $n$.

\section{Temporal translation estimates}
\label{Sec5}
In order to ensure strong
$L^2_{t,x}$ compactness of a sequence of
Faedo-Galerkin solutions, we establish
a temporal translation estimate in the space $\left(H^1\right)^\star$,
which is enough to work out the required $L^2_{t,x}$
compactness (and tightness).

\begin{lem}\label{lem:time-translation}
	Extend the Faedo-Galerkin functions
	$u_1^n(t),u_2^n(t)$, $t\in [0,T]$, which
	satisfy \eqref{S1-Galerkin-new} and
	\eqref{Galerkin-initdata}, by zero outside of $[0,T]$.
	There exists a constant $C=C(T,\Omega)>0$,
	independent of $n$, such that
	\begin{equation}\label{Time-translate}
		\E \left[\,\sup_{\abs{\tau}\in (0,\delta)}
		\norm{u_i^n(t+\tau)-u_i^n(t)}_{\left(H^1(\Om)\right)^\star}
		\, \right] \le C \delta^{1/2},
		\quad \forall t\in [0,T],
	\end{equation}
	for any sufficiently small $\delta>0$, $i=1,2$.
\end{lem}

\begin{proof}
	The aim is to estimate the expected value of
	\begin{align*}
		I(t,\tau) &:=\norm{u_1^n(t+\tau,\cdot)
			-u_1^n(t,\cdot)}_{\left(H^1(\Om)\right)^\star}
		\\ & =\sup\Set{\abs{\bigl\langle
				u_1^n(t+\tau,\cdot)-u_1^n(t,\cdot),\phi \bigr\rangle}
			\,:\, \phi\in H^1(\Om),\,
			\norm{\phi}_{H^1(\Om)}\le 1}
		\\ & =\sup\Set{\int_{\Omega}
			\bigl(u_1^n(t+\tau,x)-u_1^n(t,x)\bigr)\phi(x)\, dx
			\,:\, \phi\in H^1(\Om),\,
			\norm{\phi}_{H^1(\Om)}\le 1},
	\end{align*}
	for $\tau\in (0,\delta)$, $\delta>0$.
	Note that the same estimate can be obtained for $\tau\in (-\delta,0)$.
	Using Faedo-Galerkin approximations \eqref{eq:approxsol}, we get the following estimation
	$$
	I(t,\tau) :=\norm{u_1^n(t+\tau,\cdot)
		-u_1^n(t,\cdot)}_{\left(H^1(\Om)\right)^\star}
	\le \sum_{i=1}^4 I_i(t,\tau),
	$$
	where
	\begin{align*}
		& I_1(t,\tau)= \norm{\int_t^{t+\tau}
			\Pi_n \left[d_1\,
			\Delta u_1^n \right]
			\ds}_{\left(H^1(\Om)\right)^\star},
		\\ &
		I_2(t,\tau)= \norm{\int_t^{t+\tau}
			\Pi_n \left[\Div \Bigl (\chi(u_1^n)\nabla u_2^n\Bigr)\right]
			\ds}_{\left(H^1(\Om)\right)^\star},
		\\ &
		I_3(t,\tau)= \norm{ \int_t^{t+\tau}
			\Pi_n\left[[F_1(u_1^n,u_2^n)\right]
			\ds}_{\left(H^1(\Om)\right)^\star},
		\\ &
		I_4(t,\tau)=\norm{\sum_{k=1}^n
			\int_t^{t+\tau}\sigma^n_{u_1,k}(u_1^{n},u_2^{n})
			\dW_{u_1,k}(s)}_{\left(H^1(\Om)\right)^\star}.
	\end{align*}
	By the H\"older inequality (recall the definition of $\chi$ in \eqref{S3}),
	\begin{align*}
		\abs{\int_t^{t+\tau}\int_{\Om}
			\chi(u_1^n)\nabla u_2^n
			\cdot \nabla \Pi_n\phi\dx\ds} & \le C\,\tau^{1/2}
		\norm{\nabla u_2^n}_{L^2((0,T)\times\Om)}
		\norm{\nabla \Pi_n\phi}_{L^2(\Om)},
	\end{align*}
	for some constant $C>0$. This implies after taking the expectation and using  basic energy-type estimate \eqref{Gal:est2},
	$$
	\E \left[\abs{\int_t^{t+\tau}\int_{\Om}
		\chi(u_1^n)\nabla u_2^n
		\cdot \nabla \Pi_n\phi\dx\ds}\right] \lesssim_{T,\Om}
	\tau^{1/2}\norm{\phi}_{H^1(\Om)}.
	$$
	Consequently,
	$$
	\E \left[\sup_{0\le \tau\le \delta}
	I_2(t,\tau)\right]
	\lesssim \delta^{1/2},
	\quad \text{uniformly in $t\in [0,T]$}.
	$$
	Working exactly as $I_2$, we get
	$$
	\E \left[\sup_{0\le \tau\le \delta}
	I_1(t,\tau)\right]
	\lesssim \delta^{1/2},
	\quad \text{uniformly in $t\in [0,T]$}.
	$$
	Regarding the function $\pi$ in the defintion of $F_1$, it follows the following bound
	\begin{align*}
		\abs{\int_t^{t+\tau}\int_{\Om}
			F_1(u_1^n,u_2^n) \Pi_n\phi\,dx\,ds}& \lesssim \tau^{1/2}
		\norm{u_1^n+u_2^n}_{L^2((0,T)\times\Om)}
		\norm{\Pi_n \phi}_{L^2(\Om)}
		\\ & \
		\lesssim \tau^{1/2}\left(
		\norm{u_1^n}_{L^2((0,T)\times\Om)}^2
		+\norm{u_2^n}_{L^2((0,T)\times\Om)}^2\right)
		\norm{\phi}_{H^1(\Om)},
	\end{align*}
	where we have used Young's inequality and
	that the sequence $\Set{e_\ell}_{\ell=1}^\infty$
	is an orthonormal basis of $L^2(\Om)$, so that
	$\norm{\Pi_n \phi}_{L^2(\Om)}
	\le \norm{\phi}_{L^2(\Om)}
	\le \norm{\phi}_{H^1(\Om)}$. Hence
	$$
	\E \left[\sup_{\tau\in (0,\delta)}
	I_3(t,\tau)\right]
	\lesssim \delta^{1/2},
	\quad \text{uniformly in $t\in [0,T]$}.
	$$
	For the stochastic term $I_4$, we use the Burkholder-Davis-Gundy inequality \eqref{eq:bdg} to deduce
	\begin{align*}
		&\E \left[\sup_{\tau\in (0,\delta)}
		\norm{\sum_{k=1}^n \int_t^{t+\tau}
			\sigma^n_{u_1,k}(u_1^{n},u_2^n) \dW_{u_1,k}(s)}_{L^2(\Om)}\right]
		\\ & \qquad \lesssim
		\E \left[\sum_{k=1}^n \int_t^{t+\delta}
		\int_{\Om}\left(\sigma^n_{u_1,k}(u_1^n,u_2^n)\right)^2
		\dx\ds\right]^{\frac12}
		\\ & \qquad
		\overset{\eqref{eq:noise-cond}}{\lesssim_\Om}
		\delta^{1/2}\left(1+\E \left[
		\norm{u_1^n}_{L^\infty(0,T;L^2(\Om))}+\norm{u_2^n}_{L^\infty(0,T;L^2(\Om))}\right]\right),
	\end{align*}
	where $\E \left[\norm{u_1^n}_{L^\infty(0,T;L^2(\Om))}+\norm{u_2^n}_{L^\infty(0,T;L^2(\Om))}\right]
	\, \overset{\eqref{Gal:est3}}{\lesssim} 1$.
	As a result,
	$$
	\E \left[\sup_{\tau\in (0,\delta)}
	I_4(t,\tau)\right]
	\lesssim \delta^{1/2},
	\quad \text{uniformly in $t\in [0,T]$}.
	$$
	This concludes the proof of \eqref{Time-translate}
	for $u_1^n$. The proof for $u_2^n$ is the same.
\end{proof}

\section{Tightness and Skorokhod almost sure representations }\label{Sec6}

 Our aim in this section is to establish the tightness of
the probability measures generated
by the Faedo-Galerkin solutions
$\Set{\left(u_1^n,u_2^n,W_{u_1}^n,W_{u_2}^n,u_{1,0}^n,u_{2,0}^n\right)}_{n\ge 1}$.
We mention that the strong convergence of $u_1^n,u_2^n$ in $L_{t,x}^2$
is a consequence of the spatial $\tH$ bound \eqref{Gal:est2}
and the time translation estimate \eqref{Time-translate},
recalling that $H^1(\Om) \subset L^2(\Om) \subset
\left(H^1(\Om)\right)^\star$. We ensure
the strong (almost sure) convergence in the probability
variable $u_i\in D$ for $i=1,2$ by using some results of
Skorokhod linked to tightness (weak compactness)
of probability measures and almost sure
representations of random variables \cite{Sko-representation}.

We consider the following phase space for the probability
laws of the Faedo-Galerkin approximations:
$$
\cH:= \cH_{u_1}\times \cH_{u_2}\times \cH_{W_{u_1}}
\times \cH_{W_{u_2}}\times \cH_{u_{1,0}}\times \cH_{u_{2,0}},
$$
where
$$
\cH_{u_1},\,\cH_{u_2}=L^2(0,T;L^2(\Om))
\bigcap C\bigl(0,T;(H^1(\Om))^\star \bigr)
$$
and
$$
\cH_{W_{u_1}}, \, \cH_{W_{u_2}} = C([0,T];\U_0),
\quad
\cH_{u_{1,0}}=\cH_{u_{2,0}}= L^2(\Om).
$$
where $\U_0$ is
defined in Section \ref{sec:stoch}. We know that $\mathcal{X}_1=L^2(0,T;L^2(\Om))$,
$\mathcal{X}_2=C\bigl(0,T;(H^1(\Om))^\star \bigr)$ are
Polish spaces, therefore the intersection space $\mathcal{X}_1\cap
\mathcal{X}_2$ is Polish. Moreover, it is known products of Polish
spaces are Polish. Furthermore, since $C([0,T];\U_0)$ and
$L^2(\Om)$ are Polish, consequently $\cH$ is a Polish space. Next, we denote
 $\cB(\cH)$ the $\sigma$-algebra
of Borel subsets of $\cH$, and introduce
the measurable mapping
\begin{align*}
	&\Psi_n:\left(D,\cF,P\right)\to
	\left(\cH,\cB(\cH)\right),
	\\ & \Psi_n(\omega)
	=\bigl(u_1^n(\omega),u_2^n(\omega),W_{u_1}^n(\omega),
	W_{u_2}^n(\omega), u_{1,0}^n(\omega),u_{2,0}^n(\omega)\bigr).
\end{align*}
Now we define a probability measure $\cL_n$
on $\left(\cH,\cB(\cH)\right)$ by
\begin{equation}\label{eq:prob-measure-n}
	\cL_n(\cA)= \left(P\circ \Psi^{-1}\right)(A)
	=P\left(\Psi_n^{-1}(\cA)\right),
	\quad \cA\in \cB(\cH).
\end{equation}
Denote by $\cL_{u_1^n}$, $\cL_{u_2^n}$,
$\cL_{W_{u_1}^n}$, $\cL_{W_{u_2}^n}$,
$\cL_{u_{1,0}^n}$, $\cL_{u_{2,0}^n}$
the respective laws of $u_1^n$, $u_2^n$, $W_{u_1}^n$,
$W_{u_2}^n$, $u_{1,0}^n$ and $u_{2,0}^n$, which are defined respectively
on $\left(\cH_{u_1},\cB(\cH_{u_1})\right)$,
$\left(\cH_{u_2},\cB(\cH_{u_2})\right)$,
$\left(\cH_{W_{u_1}},\cB(\cH_{W_{u_1}})\right)$,
$\left(\cH_{W_{u_2}},\cB(\cH_{W_{u_2}})\right)$
$\left(\cH_{u_{1,0}},\cB(\cH_{u_{1,0}})\right)$ and
$\left(\cH_{u_{2,0}},\cB(\cH_{u_{2,0}})\right)$. Therefore
$$
\cL_n = \cL_{u_1^n}\times \cL_{u_2^n}
\times  \cL_{W_{u_1}^n}\times  \cL_{W_{u_2}^n}
\times  \cL_{u_{1,0}^n}\times \cL_{u_{2,0}^n}.
$$

We give sequences $\Set{r_m}_{m\ge1}, \Set{\nu_m}_{m\ge1}$
of positive numbers tending to zero as $m\to \infty$ and we introduce the following Banach space
\begin{align*}
	\cZ_{r_m,\nu_m}
	:=\Biggl\{ z \in & L^\infty\left(0,T;L^2(\Om)\right)
	\cap L^2\left(0,T;\tH(\Om)\right)
	\, : \,
	\\ & \sup_{m\ge 1} \frac{1}{\nu_m}
	\sup_{\tau \in (0,r_m)}
	\norm{z(\cdot+\tau)-z}_{L^\infty\left(0,T-\tau;
		(H^1(\Om))^\star\right)}
	<\infty \Biggr\},
\end{align*}
under the norm
\begin{align*}
	\norm{z}_{\cZ_{r_m,\nu_m}}
	:= &  \norm{z}_{L^\infty\left(0,T;L^2(\Om)\right)}
	+\norm{z}_{L^2\left(0,T;\tH(\Om)\right)}
	\\ & \qquad
	+ \sup_{m\ge 1} \frac{1}{\nu_m}\sup_{0\le \tau\le r_m}
	\norm{z(\cdot+\tau)-z}_{L^\infty\left(0,T-\tau;
		\left(H^1(\Om)\right)^\star\right)}.
\end{align*}

According to \cite{Simon:1987vn}, We have the following compact embedding (consult \cite[Theorem 5]{Simon:1987vn})
$$
\cZ_{r_m,\nu_m}\subset\subset
L^2(0,T;L^2(\Om))\cap
C\bigl([0,T];(H^1(\Om))^\star\bigr).
$$

\noindent We have the following regarding the tightnees of the laws $\cL_n$, cf.~\eqref{eq:prob-measure-n}.

\begin{lem}\label{lem:tight}
	The sequence $\Set{\cL_n}_{n\ge1}$ of
	probability measures is (uniformly) tight, and
	therefore weakly compact,
	on the phase space $\left(\cH,\cB(\cH)\right)$.
\end{lem}

\begin{proof}
	In our proof, we produce compact sets (for each $\delta>0$)
	\begin{align*}
		& \C_{1,\delta}\subset L^2(0,T;L^2(\Om))
		\bigcap C\bigl(0,T;(H^1(\Om))^\star \bigr),
		\\ &
		\text{and} \quad \C_{2,\delta} \subset C([0,T];\U_0),
		\quad
		\C_{3,\delta} \subset L^2(\Om),
	\end{align*}
	such that $\cL_n\left(\C_{\delta}\right)
	=P\left(\Set{\Phi_n \in \C_{\delta}} \right)> 1-\delta$,
	where $\C_{\delta}:=\left(\C_{1,\delta}\right)^2
	\times \left(\C_{2,\delta}\right)^2
	\times \left(\C_{3,\delta}\right)^2$.
	We show that\\
	$\cL_n\left(\C_{i,\delta}^c\right)\le \delta/6$
	for $i=1,2,3$.
	For this, we take the sequences
	$\Set{r_m}_{m=1}^\infty$, $\Set{\nu_m}_{m=1}^\infty$
	such that
	\begin{equation}\label{eq:r-nu}
		\sum_{m=1}^\infty
		\frac{r_m^{1/4}}{\nu_m}<\infty,
	\end{equation}
	and
	$$
	\C_{1,\delta}:=
	\Set{z\in \cZ_{r_m,\nu_m}:
		\norm{z}_{\cZ_{r_m,\nu_m}}\le R_{1,\delta}},
	$$
	where $R_{1,\delta}>0$ is a number
	to be determined later.\\
\noindent Now, we use \cite[Theorem 5]{Simon:1987vn} to deduce that $\C_{1,\delta}$ is a
	compact subset of $L^2(0,T;L^2(\Om))$.
	For $i=1,2$, we have
	\begin{align*}
		& P\left(\Set{u_i\in D:
			u_i^n(u_i)\notin\C^{1,\delta}} \right)
		\\ & \quad \le P\left(\Set{u_i\in D:
			\norm{u_i^n(u_i)}_{L^\infty\left(0,T;L^2(\Om)\right)}
			>R_{1,\delta}}\right)
		\\ & \quad\qquad
		+P\left(\Set{u_i\in D:
			\norm{u_i^n(u_i)}_{L^2\left(0,T;\tH(\Om)\right)}
			>R_{1,\delta}}\right)
		\\ & \quad\qquad +P\left(\Set{u_i\in D:
			\sup_{\tau\in (0,r_m)} \norm{u_i^n(\cdot+\tau)
				-u_i^n}_{L^\infty\left(0,T-\tau;
				\left(H^1(\Om)\right)^\star\right)}
			>R_{1,\delta} \, \nu_m}\right)
		\\ & \quad
		=: P_{1,1}+P_{1,2}+P_{1,3}
		\quad \text{(for any $m\ge 1$)}.
	\end{align*}
	An application of the Chebyshev inequality, we deduce
	\begin{align*}
		P_{1,1} & \le \frac{1}{R_{1,\delta}}
		\E \left[\norm{u_i^n(u_i)}_{L^\infty\left(0,T;
			L^2(\Om)\right)}\right]\le \frac{C}{R_{1,\delta}},
		\\ P_{1,2} & \le \frac{1}{R_{1,\delta}}
		\E\left[\norm{u_i^n(u_i)}_{L^2\left(0,T;
			\tH(\Om)\right)}\right]
		\le \frac{C}{R_{1,\delta}},
		\\ P_{1,3} & \le \sum_{m=1}^\infty
		\frac{1}{R_{1,\delta} \, \nu_m}
		\E\left[\, \sup_{0\le \tau \le r_m}
		\norm{u_i^n(\cdot+\tau)-u_i^n}_{L^\infty
			\left(0,T-\tau;\left(H^1(\Om)\right)^\star\right)}\right]
		\\ & \le \frac{C}{R_{1,\delta}}
		\sum_{m=1}^\infty \frac{r_m^{1/4}}{\nu_m}
		\overset{\eqref{eq:r-nu}}{\le}
		\frac{C}{R_{1,\delta}}.
	\end{align*}
	Herein, we have used \eqref{Gal:est2}, \eqref{Gal:est3},
	and \eqref{Time-translate}.
	We can choose $R_{1,\delta}$ such that
	$$
	\cL_{u_i^n}\left(\C_{1,\delta}^c\right)
	=P\left(\Set{u_i\in D:
		u_i^n(u_i)\notin\C_{1,\delta}} \right)
	\leq \frac{\delta}{6}, \quad i=1,2.
	$$
	We know
	that the finite series $W_{u_1}^{n},W_{u_2}^{n}$
	are $P$-a.s.~convergent in $C([0,T];\U_0)$ as $n\to \infty$.
	Consequently the laws $\cL_{W_{u_1}^n}, \cL_{W_{u_2}^n} $
	converge weakly. Now, we use Prokhorov's weak
	compactness characterization (see
	e.g.~\cite[Theorem 2.3]{DaPrato:2014aa})) to deduce the tightness of $\Set{\cL_{W^n_{u_1}}}_{n\ge 1}$ and
	$\Set{\cL_{W^n_{u_2}}}_{n\ge 1}$. Therefore, for any
	$\delta>0$, there exists a compact
	set $\C_{2,\delta}$ in $C([0,T];\U_0)$ such that
	$$
	\cL_{W_{u_i}^n}\left(\C_{2,\delta}^c\right)
	=P\left(\Set{u_i\in D: W_{u_i}^n(u_i)
		\notin\C_{2,\delta}} \right)\leq \frac{\delta}{6},
	\quad i=1,2.
	$$
	Moreover, the initial data approximations
	$u_{1,0}^n, u_{2,0}^n$ are $P$-a.s.~convergent in $L^2(\Om)$
	as $n\to \infty$ and the laws $\cL_{u_{1,0}^n}, \cL_{u_{2,0}^n}$
	converge weakly (with $\cL_{u_{1,0}^n}\weak \mu_{u_{1,0}}$,
	$\cL_{u_{2,0}^n}\weak u_{2,0}^n$). This implies that
	these laws are tight and
	$$
	\cL_{w_0^n}\left(\C_{3,\delta}\right)
	=P\left(\Set{u_i\in D: w_0^n(u_i)
		\notin\C_{3,\delta}} \right)
	\leq \frac{\delta}{6}, \quad i=1,2.
	$$
	This implies that $\Set{\cL_n}_{n\ge1}$
	is a tight sequence of probability measures.
	The weak compactness of $\Set{\cL_n}_{n\ge1}$ is the consequence of Prokhorov's theorem \cite[Theorem 2.3]{DaPrato:2014aa}.
\end{proof}

Note that the probability measures $\cL_n$ form a sequence that is weakly compact on $\left(\cH,\cB(\cH)\right)$. As result,  we deduce that $\cL_n$ converges weakly to a probability measure $\cL$ on $\cH$ (up to a subsequence).  Now, we can apply the Skorokhod  theorem
(see e.g.~\cite[Theorem 2.4]{DaPrato:2014aa})
to deduce the existence of a new
probability space $(\tilde{D},\tilde{\cF}, \tilde{P})$
and new random variables
\begin{equation}\label{eq:def-tvar}
	\begin{split}
		\tilde \Psi_n=\left(\tilde u_1^n, \tilde u_2^n,
		\tilde W_{u_1}^n, \tilde W_{u_2}^n,
		\tilde u_{1,0}^n, \tilde u_{2,0}^n\right),
		\quad
		\tilde \Psi=\left (\tilde u_1, \tilde u_2,
		\tilde W_{u_1}, \tilde W_{u_2},\tilde u_{1,0},\tilde u_{2,0}\right),
	\end{split}
\end{equation}
with respective joint laws $\tilde \cL_n=\cL_n$
and $\tilde \cL=\cL$, such that $\tilde \Psi_n \to \tilde \Psi$
almost surely in the topology of $\cX$. Thus, the
following convergences hold $\tilde P$-almost
surely as $n\to \infty$:
\begin{equation}\label{eq:strong-conv0}
	\begin{split}
		& \tilde u_1^n \to \tilde u_1,
		\quad  \tilde u_2^n \to \tilde u_2  \quad
		\text{in $L^2(0,T;L^2(\Om))$},
		\\ & \tilde u_1^n\to \tilde u, \quad
		\tilde u_2^n \to \tilde u_2
		\quad \text{in $C\bigl([0,T];
			\left(H^1(\Om)\right)^\star\bigr)$},
		\\ &
		\tilde W_{u_1}^{n} \to \tilde W_{u_1},
		\quad
		\tilde W_{u_2}^{n} \to \tilde W_{u_2}
		\quad \text{in $C([0,T]; \U_0)$},
		\\ &
		\tilde u_{1,0}^n \to \tilde u_{1,0},
		\quad
		\tilde u_{2,0}^n\to \tilde u_{2,0}
		\quad \text{in $L^2(\Om)$.}
	\end{split}
\end{equation}

Observe that by equality of the laws, the estimates
in Lemma \ref{lem:apriori-est}
and \eqref{eq:Lq0-estimation} continue to hold for
the new random variables $\tilde u_i^n$ ($i=1,2$).
Moreover, all estimates for the
Faedo-Galerkin approximations $u_i^n$ are valid for the
"tilde" approximations $\tilde u_i^n$
defined on the new probability space
$(\tilde{D},\tilde{\cF}, \tilde{P})$.
Additionally, we have for any $q\in [2,q_0]$ (recall that $q_0>3$),
\begin{equation}\label{eq:Lq0-est-tilde}
	\tilde\E\left[
	\norm{\tilde u_i^n}_{L^\infty(0,T;L^2(\Om)}^q
	\, \right]\le C,
	\quad
	\tilde \E\left[
	\norm{\Grad \tilde u_i^n}_{L^2((0,T)\times \Om)}^q
	\right]\le C,
	\quad i=1,2,
\end{equation}
where the constant $C$ is independent of $n$.

Now, we consider the stochastic basis
\begin{equation}\label{eq:stoch-basis-n}
	\tilde \cS_n=\bigl(\tilde D,\tilde \cF,
	\bigl\{\tilde \cF_t^n\bigr\}_{t\in [0,T]},\tilde P,
	\tilde W_{u_1}^{n},\tilde W_{u_2}^{n}\bigr),
\end{equation}
where
$$
\tcFt^n = \sigma\bigl(
\sigma \bigl(\tilde \Psi_n\big |_{[0,t]}\bigr)
\bigcup \bigl\{N \in \tilde \cF: \tilde P (N)=0\bigr\}\bigr).
$$
The filtration $\bigl\{\tcFt^n\bigr\}_{n\ge 1}$
is the smallest such that  the "tilde processes"
$\tilde u_1^n$, $\tilde u_2^n$, $\tilde W_{u_1}^n$, $\tilde W_{u_2}^n$,
$\tilde u_{1,0}^n$, and $\tilde u_{2,0}^n$ are adapted.

In view of equality of the laws and
L\'{e}vy's martingale characterization of a Wiener process,
see \cite[Theorem 4.6]{DaPrato:2014aa}, we conclude that
$\tilde W_{u_1}^{n}$ and $\tilde W_{u_2}^{n}$ are cylindrical Wiener
processes. Moreover, we claim that $\tilde W_{u_1}^n$, $\tilde W_{u_2}^n$
are cylindrical Wiener processes relative to
the filtration $\bigl\{\tcFt^n\bigr\}_{n\ge 1}$
defined in \eqref{eq:stoch-basis-n}.
To prove this, we verify that $\tilde W_{u_i}^n(t)$
is $\tcFt^n$ measurable and $\tilde W_{u_i}^n(t)-\tilde W_{u_i}^n(s)$
is independent of $\tcFs^n$, for all $0\le s< t\le T$,
$i=1,2$.  Since $\tilde W_{u_i}^n$ and $W_{u_i}^n$ have
the same laws and that $W_{u_i}^n(t)$ is $\cFt$
measurable and $W{u_i}^n(t)-W_{u_i}^n(s)$ is
independent of $\cFs$, we obtain the aforesaid properties.

Thus, there exist sequences (recall that $\bigl\{\psi_k\bigr\}_{k\ge 1}$ is the basis
of $\U$ and the series converge in $\U_0 \supset \U$)
$\bigl\{\tilde W_{u_1,k}^{n}\bigr\}_{k\ge 1}$,
$\bigl\{\tilde W_{u_2,k}^{n}\bigr\}_{k\ge 1}$ of
mutually independent real-valued Wiener processes
adapted to $\bigl\{\tcFt^n\bigr\}_{t\in [0,T]}$ such that
\begin{equation}\label{eq:Wiener-tilde-n}
	\tilde W_{u_i}^n=\sum_{k\ge 1}\tilde W_{u_i,k}^n \psi_k,
	\quad \text{for $i=1,2$}.
\end{equation}
Next, we will use the following $n$-truncated sums
$$
\tilde W_{u_i}^{(n)} = \sum_{k=1}^n\tilde W_{u_i,k}^n\psi_k,
\quad i=1,2,
$$
which converges to $\tilde W_{u_i}$
in $C([0,T];\U_0)$, $\tilde P$-almost surely for $i=1,2$.

Using \eqref{eq:approx-eqn-integrated} and equality of the laws, the following equations hold $\tilde P$-almost surely on the new probability space
$\bigl(\tilde{D},\tilde{\cF}, \tilde{P}\bigr)$:
\begin{equation}\label{eq:approx-eqn-tilde}
	\begin{split}
		& \tilde u_1^n(t)-\int_0^t
		\Pi_n \left[d_1\, \Delta \tilde u_1^n \right]\ds
		+\int_0^t \Pi_n \left[\Div \left(\chi(\tilde u_1^n)\nabla \tilde u_2^n\right)\right]\ds
		\\ & \qquad \qquad \qquad
		= \tilde u_{1,0}^n+\int_0^t
		\Pi_n\left[F_1(\tilde u_1^n, \tilde u_2^n)\right]\ds
		+\int_0^t\sigma^n_{u_1}(\tilde u_1^n)
		\,d \tilde W^{(n)}_{u_1}(s)
		\quad \text{in $L^2(\Om)$},
		\\ &
		\tilde u_2^n(t)-\int_0^t
		\Pi_n \left[d_2\, \Delta \tilde u_2^n \right]\ds
		\\ & \qquad \qquad \qquad
		= \tilde u_{2,0}^n+\int_0^t \Pi_n\left[
		F_2(\tilde u_1^n,\tilde u_2^n)\right]\ds
		+\int_0^t \sigma^n_{u_2}(\tilde u_2^n)
		\,d\tilde W^{(n)}_{u_2}(s)
		\quad \text{in $L^2(\Om)$},
	\end{split}
\end{equation}
for any $t\in [0,T]$, where $\sigma^n_{u_i}(\tilde u_i^n)
\, d\tilde W^{(n)}_{u_i}=\sum_{k=1}^n \sigma^n_{u_i,k}(\tilde u_i^n)
\, d\tilde W^n_{u_i,k}$ for $i=1,2$.

\medskip

\section{Passing to the limit in the Faedo-Galerkin equations}\label{Sec7}

We will need a stochastic basis for the limit of
the Skorokhod representations, i.e., for the
variables \\$\tilde \Psi:=\bigl(\tilde u_1, \tilde u_2, \tilde W_{u_1},
\tilde W_{u_2},\tilde u_{0,1}, \tilde u_{0,2}\bigr)$,
cf.~\eqref{eq:def-tvar}: specifically,
\begin{equation}\label{eq:stoch-basis-tilde}
	\tilde \cS=\bigl(\tilde D,\tilde \cF,
	\bigl\{\tilde \cF_t\bigr\}_{t\in [0,T]},\tilde P,
	\tilde W_{u_1},\tilde W_{u_2}\bigr),
\end{equation}
where $\tcFt = \sigma\bigl(\sigma
\bigl(\tilde \Psi\big |_{[0,t]}\bigr)
\bigcup \bigl\{N \in \tilde \cF: \tilde P (N)=0\bigr\}\bigr)$.
We know that $\tilde W_{u_1}^{n}$, $\tilde W_{u_2}^{n}$
are cylindrical Wiener processes with
respect to $\tilde \cS_n$ (see \eqref{eq:stoch-basis-n}
and \eqref{eq:Wiener-tilde-n}) and
$\tilde W_{u_1}^{n}\to \tilde W_{u_1}$,
$\tilde W_{u_2}^{n} \to \tilde W_{u_2}$ in
the sense of \eqref{eq:strong-conv0}.
Consequently, there exist sequences
$\bigl\{\tilde W_{u_1,k}\bigr\}_{k\ge 1}$,
$\bigl\{\tilde W_{u_2,k}\bigr\}_{k\ge 1}$
of real-valued Wiener processes
adapted to the filtration
$\bigl\{\tilde \cFt\bigr\}_{t\in [0,T]}$,
cf.~\eqref{eq:stoch-basis-tilde}, such that
$\tilde W_{u_1}=\sum_{k\ge 1}\tilde W_{u_1,k} \psi_k$
and $\tilde W_{u_2}=\sum_{k\ge 1}\tilde W_{u_2,k} \psi_k$.

Exlpoiting the estimations \eqref{eq:Lq0-est-tilde} and the a.s.
convergences in \eqref{eq:strong-conv0},
we deduce by passing
if necessary to subsequence as $n\to \infty$
\begin{equation}\label{eq:tilde-conv}
	\begin{split}
		\mathrm{i)} & \quad \tilde u_1^n\to \tilde u_1, \quad
		\tilde u_2^n \to \tilde u_2
		\quad
		\text{in $L^2\bigl(\tilde D,\tilde \cF,\tilde P;
			L^2(0,T;L^2(\Om))\bigr)$},
		\\
		\mathrm{ii)} & \quad
		\tilde u_1^n\weak \tilde u_1, \quad
		\tilde u_2^n \weak \tilde u_2
		\quad
		\text{in $L^2\bigl(\tilde D,\tilde \cF,\tilde P;
			L^2(0,T;\tH(\Om))\bigr)$},
		\\
		\mathrm{iii)} & \quad
		\tilde u_1^n \weakstar \tilde u_1, \quad
		\tilde u_2^n \weakstar \tilde u_2
		\quad
		\text{in $L^2\bigl(\tilde D,\tilde \cF,\tilde P;
			L^\infty(0,T;L^2(\Om))\bigr)$},
		\\
		\mathrm{iv)} & \quad
		\tilde u_1^n\to \tilde u_1, \quad
		\tilde u_2^n \to \tilde u_2
		\quad \text{in $L^2\bigl(\tilde D,\tilde \cF,
			\tilde P;C\bigl([0,T];\left(H^1(\Om)\right)^\star
			\bigr)\bigr)$},
		\\
		\mathrm{v)} & \quad
		\tilde W_{u_1}^n \to \tilde W_{u_1},
		\quad
		\tilde W_{u_2}^n \to \tilde W_{u_2}
		\quad
		\text{in $L^2\bigl(\tilde D,\tilde \cF,
			\tilde P;C([0,T];\U_0)\bigr)$},
		\\
		\mathrm{vi)} & \quad
		\tilde u_{1,0}^n \to \tilde u_{1,0},
		\quad
		\tilde u_{2,0}^n \to \tilde u_{2,0}
		\quad
		\text{in $L^2\bigl(\tilde D,\tilde \cF,
			\tilde P;L^2(\Om)\bigr)$.}
	\end{split}
\end{equation}

Finally, we pass to the limit in the
Faedo-Galerkin equations \eqref{eq:approx-eqn-tilde}.

\begin{lem}[limit equations]\label{lem:limit-eqs}
	The limits $\tilde u_1$, $\tilde u_2$, $\tilde W_{u_1}$,
	$\tilde W_{u_2}$, $\tilde u_{1,0}$, $\tilde u_{2,0}$ of
	the Skorokhod a.s.~representations of the Faedo-Galerkin
	approximations---constructed in
	\eqref{eq:def-tvar}, \eqref{eq:strong-conv0}---satisfy
	the following equations $\tilde P$-a.s.,
	for all $t\in [0,T]$:
	\begin{align}
		&\int_{\Om} \tilde u_1(t) \vphi_{u_1} \dx
		-\int_{\Om} \tilde u_{1,0} \, \vphi_{u_1} \dx +\int_0^t \int_{\Om}
		\left (d_1 \Grad \tilde u_1 - \chi(\tilde u_1) \nabla \tilde u_2\right)
		\cdot \Grad \vphi_{u_1}  \dx\ds
		\notag \\
		 & \quad \quad
		= \int_0^t \int_{\Om}F_1(\tilde u_1, \tilde u_2)\vphi_{u_1}  \dx\ds
		+ \int_0^t \int_{\Om} \sigma_{u_1}(\tilde u_1,\tilde u_2)
		\vphi_{u_1} \dx  \, d \tilde W_{u_1} (s),
		\label{eq:weakform-u-tilde}
		\\ &
		\int_{\Om} \tilde u_2(t)  \vphi_{u_2}\dx
		-\int_{\Om} \tilde u_{2,0} \vphi_{u_2}\dx+ \int_0^t \int_{\Om}
		d_2 \Grad \tilde u_2
		\cdot \Grad \vphi_{u_2} \dx\ds
		\notag \\ & \quad \quad
		=\int_0^t \int_{\Om} F_2(\tilde u_1, \tilde u_2) \vphi_{u_2}\dx\ds
		+ \int_0^t \int_{\Om} \sigma_{u_2}(\tilde u_1,\tilde u_2)
		\vphi_{u_2} \dx \, d \tilde W_{u_2}(s),
		\label{eq:weakform-v-tilde}
	\end{align}
	for all $\vphi_{u_1},\vphi_{u_2} \in H^1(\Omega)$,
	where the laws of $\tilde u_{1,0}$
	and $\tilde u_{2,0}$ are $\mu_{u_{1,0}}$
	and $\mu_{u_{2,0}}$, respectively.
\end{lem}

\begin{proof}
	First, we fix $\vphi_{u_i}\in H^1(\Om)$, and
	we write \eqref{eq:weakform-u-tilde}-\eqref{eq:weakform-v-tilde}
	symbolically as $I_{u_i}(\omega,t)=0$, for
	$(\omega,t)\in \tilde D\times (0,T)$ and for $i=1,2$. Our goal is to demonstrate that for $i=1,2$
	$$
	\norm{I_{u_i}}_{L^2(\tilde D\times (0,T))}^2
	=\tilde{\E} \int_0^T \left(I_{u_i}(\omega,t)
	\right)^2\dt=0,
	$$
	which implies that $I_{u_i}=0$ for
	$d\tilde P\times dt$-a.e.~$(\omega,t)\in
	\tilde D\times (0,T)$ and thus, by the
	Fubini theorem, $I_{u_i}=0$ $\tilde P$-a.s.,
	for a.e.~$t\in (0,T)$. By density in $L^2$, we prove that for $i=1,2$
	\begin{equation}\label{eq:limit-weak-form-tilde-tmp0}
		\E \left[\,\int_0^T \En_Z(\omega,t)
		I_{u_i}(\omega,t)\right] \,dt=0,
	\end{equation}
	for a measurable set
	$Z \subset \tilde D\times (0,T)$,
	where $\En_Z(\omega,t)\in
	L^\infty \left(\tilde D\times (0,T);
	\tilde dP\times dt \right)$
	denotes the characteristic function of $Z$.\\
\noindent	We multiply \eqref{eq:weakform-u-tilde} with
	$\vphi_{u_1}\in H^1(\Om)$, we intgrate by parts and we use the basic properties
	of the projection operator $\Pi_n$ to obtain
	\begin{equation}\label{eq:weakform-tilde-tmp1}
		\begin{split}
			& \int_{\Om} \tilde u_1^n(t)\vphi_{u_1} \dx
			+ \int_0^t \int_{\Om} d_1
			\Grad \tilde u_1^n
			\cdot \Grad \Pi_n \vphi_{u_1} \dx\ds- \int_0^t \int_{\Om} \chi(\tilde u_1^n)  \Grad \tilde u_2^n\cdot
			\Grad \Pi_n \vphi_{u_1} \dx\ds
			\\ &
			= \int_{\Om} \tilde u_{1,0}^n \vphi_{u_1} \dx
			+ \int_0^t \int_{\Om} F_1(\tilde u_1^n,\tilde u_2^n)
			\Pi_n \vphi_{u_1} \dx\ds
			+\int_0^t \int_{\Om} \sigma_{u_1}^n(\tilde u_1^n,\tilde u_2^n)
			\Pi_n  \vphi_{u_1} \dx
			\, d\tilde W_{u_1}^{(n)}(s).
		\end{split}
	\end{equation}
	Next, we multiply \eqref{eq:weakform-tilde-tmp1}
	with the characteristic function (of $Z$)
	$\En_Z(\omega,t)$, we integrate the result
	over $(\omega,t)$, and then
	we pass to the limit $n\to\infty$ in each term separately.
	
	Now, we use part vi) of \eqref{eq:tilde-conv} to get (recall
	that $u_{1,0}^n=\Pi_n u_{1,0} \to u_{1,0}$ in $L^2(\Om)$
	and $u_{1,0}\sim \mu_{u_{1,0}}$)
	$$
\displaystyle\tilde\E\int_0^T\int_{\Om} \En_Z \tilde u_{1,0}^n
	\vphi_{u_1} \dx\overn\tilde\E\int_0^T\int_{\Om}
	\En_Z \tilde u_{1,0} \vphi_{u_1} \dx.
$$
	Since the laws of $u_{1,0}^n$ and $\tilde{u}_{1,0}^n$
	are the same, we dedude that $\tilde u_{1,0}\sim \mu_{u_{1,0}}$.
	
	
	Note that, the weak convergence in $L^2_{\omega,t,x}$ of
	$\tilde \Grad u_1^n$
	implies that (consult \eqref{eq:tilde-conv}--(ii))
	\begin{align*}
		&\tilde \E \left[\int_0^T \En_Z(\omega,t)
		\left( \int_0^t
		\int_{\Om} d_1 \Grad \tilde u_1^n
		\cdot \Grad \Pi_n \vphi_{u_1} \dx\ds
		\right)\dt \right]
		\\ & \qquad
		\overn \tilde \E \left[\int_0^T \En_Z(\omega,t)
		\left(\int_0^t \int_{\Om}
		d_1
		\Grad \tilde u_1\cdot
		\Grad \vphi_{u_1} \dx\ds\right)\dt \right].
	\end{align*}
	For the prey-taxis term exploit the convergences:
	$\chi (\tilde u_1^n) \Grad \Pi_n \vphi_{u_1}
	\overn \chi (\tilde u_1) \Grad \vphi_{u_1}$ strongly in $L^2_{\omega,t,x}$,
	$\Grad\Pi_n \vphi_{u_1}\to \Grad \vphi_{u_1}$ in $L^2_x$ and
         the strong $L^2_{\omega,t,x}$ convergence
	of $\tilde u_1^n$. The result is
	\begin{align*}
		&\E \left[\int_0^T \En_Z(\omega,t)
		\left(\int_0^t  \int_{\Om}
		\chi(\tilde u_1^n)\nabla \tilde u_2^n
		\cdot \Grad \Pi_n \vphi_{u_1}
		\dx\ds\right) \dt\right]
		\\ & \qquad
		\overn
		\E \left[\int_0^T \En_Z(\omega,t)
		\left(\int_0^t\int_{\Om}
		\chi(\tilde u_1) \nabla \tilde u_2
		\cdot \Grad \vphi_{u_1} \dx\ds\right) \dt\right].
	\end{align*}
	
	Recalling that the function $F_1$ is globally Lipschitz and $\Pi_n\vphi_{u_1}\to \vphi_{u_1}$ in $L^2(\Om)$,
	we deduce from the strong convergences
	$\tilde u_i^n\to \tilde u_i$ in $L^2_{\omega,t,x}$ for $i=1,2$ (consult \eqref{eq:tilde-conv}--(i))
	\begin{align*}
		& \E \left[\int_0^T \En_Z(\omega,t)
		\left (\int_0^t \int_{\Om}
		F_1(\tilde u_1^n,\tilde u_2^n) \Pi_n\vphi_{u_1}
		\dx\ds\right)\dt \right]
		\\ & \qquad
		\overn
		\E \left[\int_0^T \En_Z(\omega,t)
		\left (\, \int_0^t \int_{\Om}
		F_1(\tilde u_1,\tilde u_2) \vphi_{u_1} \dx\ds\right)
		\dt \right].
	\end{align*}	
	Regarding the stochastic integral, we prove first  that
	\begin{equation}\label{eq:stoch-conv-tmp1}
		\int_0^t \sigma_{u_1}^n(\tilde u_1^n,\tilde u_2^n)
		\, d\tilde W_{u_1}^{(n)}(s) \overn
		\int_0^t \sigma_u(\tilde u_1,\tilde u_2)
		\, d\tilde W_{u_1} (s)
		\quad \text{in $L^2\left(0,T;L^2(\Om)\right)$},
	\end{equation}
	in probability (with respect to $\tilde P$).
	Since $\tilde W_{u_1}^{(n)} \to \tilde W_{u_1}$ in $C\big([0,T];\U_0\big)$,
	$\tilde P$-a.s.~and thus in probability,
	cf.~\eqref{eq:strong-conv0}, it remains to prove that
	\begin{equation}\label{eq:etan-conv}
		\sigma_{u_1}^n(\tilde u_1^n,\tilde u_2^n) \to \sigma_{u_1}(\tilde u_1,\tilde u_2)
		\quad \text{in $L^2\bigl(0,T;
			L_2(\U;L^2(\Om))\bigr)$,
			$\tilde P$-almost surely}.
	\end{equation}
	
	Clearly,
	\begin{equation}\label{est:vitali-start1}
		\begin{split}
			&\int_0^T
			\norm{\sigma_{u_1}(\tilde u_1,\tilde u_2)
				-\sigma_{u_1}^n(\tilde u_1^n,\tilde u_2^n)}_{L_2\left(\U;L^2(\Om)\right)}^2
			\dt \\ & \quad
			\leq \int_0^T \norm{\sigma_{u_1}(\tilde u_1,\tilde u_2)
				-\sigma_{u_1}(\tilde u_1^n,\tilde u_2^n)}_{L_2\left(\U;L^2(\Om)\right)}^2 \dt
			\\ & \quad \qquad
			+\int_0^T \norm{\sigma_{u_1}(\tilde u_1,\tilde u_2)
				-\sigma_{u_1}^n(\tilde u_1,\tilde u_2)}_{L_2\left(\U;L^2(\Om)\right)}^2 \dt
			=: I_1+I_2.
		\end{split}
	\end{equation}
	Using \eqref{eq:noise-cond2} and \eqref{eq:strong-conv0},
	we obtain easily
	\begin{equation}\label{eq:noise-conv-I1}
		I_1 \overn 0,
		\quad \text{$\tilde P$-almost surely}.
	\end{equation}
	For  $I_2$, we have (recall the definitions of $\sigma_{u_1,k}$, $\sigma_{u_1,k,\ell}$ defined respectively in \eqref{def:sint}, \eqref{def-sigwkl})
	\begin{align*}
		I_2 &= \int_0^T \sum_{k \geq 1}
		\norm{\sigma_{u_1,k}(\tilde u_1,\tilde u_2)
			-\sigma_{u_1,k}^n(\tilde u_1,\tilde u_2)}_{L^2(\Om)}^2 \dt
		\\ & = \int_0^T \sum_{k \geq 1}
		\norm{\sigma_{u_1,k}(\tilde u_1,\tilde u_2)
			-\Pi_n\bigl(\sigma_{u_1,k}(\tilde u_1,\tilde u_2)
			\bigr)}_{L^2(\Om)}^2 \dt
		=:\int_0^T \mathcal{I}_n(t)\dt.
	\end{align*}
Moreover, we have the following bound ($\tilde P$-a.s.) (recall that $\tilde u_i \in L^2_{\omega} L^\infty_tL^2_x$ for $i=1,2$ (a.s.))
	\begin{align*}
		\displaystyle0\leq \mathcal{I}_n(t) & \displaystyle\le 4\sum_{k \geq 1}
		\norm{\sigma_{u_1,k}(\tilde u_1(t),\tilde u_2(t))}^2_{L^2(\Om)}
		= 4
		\norm{\sigma_{u_1}(\tilde u_1(t),\tilde u_2(t))}^2_{L_2\left(\U;L^2(\Om)\right)}
		\\ & \displaystyle\overset{\eqref{eq:noise-cond2}}{\le}
		C\left(1+\norm{\tilde u_1(t)}_{L^2(\Om)}^2+\norm{\tilde u_2(t)}_{L^2(\Om)}^2 \right)
		\in L^1(0,T)\qquad \text{$\tilde P$-a.s.}.
	\end{align*}
This implies that\\ $\norm{
		\sigma_{u_1}(\tilde u_1,\tilde u_2)}_{L_2\left(\U;L^2(\Om)\right)}^2
	\in L^1_t$ a.s.~and $\sum_{k\ge1}
	\abs{\sigma_{u_1,k}(\tilde{u}_1,\tilde u_2)}^2
	\in L^1_{t,x}$ a.s., thus
	$$
	\Pi_n\left(\sum_{k \geq 1}
	\sigma_{u_1,k}(\tilde u_1,\tilde u_2)\right)
	\overn \sum_{k \geq 1} \sigma_{u_1,k}(\tilde u_1,\tilde u_2)
	\quad \text{in $L^2(\Om)$},
	$$
	for a.e.~$t$ and almost surely.
	Using this,
	$$
	\mathcal{I}_n(t)\overn 0, \quad
	\text{a.e.~on $[0,T]$ (and a.s)},
	$$
	and an application of Lebesgue's dominated
	convergence theorem, we arrive to
	\begin{equation}\label{eq:noise-conv-I2}
		I_2 \overn 0,
		\quad \text{$\tilde P$-almost surely}.
	\end{equation}
	The convergence  \eqref{eq:etan-conv} is a consequence of \eqref{est:vitali-start1},
	\eqref{eq:noise-conv-I1} and \eqref{eq:noise-conv-I2}. Therefore we obtain
\eqref{eq:stoch-conv-tmp1}.
	
Next, we fix any number $q\in (2,q_0]$ (consult \eqref{cond-init}), we use
	Burkholder-Davis-Gundy inequality \eqref{eq:bdg}
	and \eqref{eq:noise-cond}, \eqref{eq:Lq0-est-tilde}
	to obtain
	\begin{align*}
		&\tilde \E\left[\, \norm{\int_0^t \sigma_{u_1}^n(\tilde u_1^n,\tilde u_2^n)
			\, d \tilde W_{u_1}^{(n)}}_{L^2((0,T);L^2(\Om))}^q \, \right]
		\\ & \quad
		\le \bar C_T  \tilde \E\left[\, \sup_{t\in[0,T]}
		\norm{\sum_{k=1}^n\int_0^t \sigma_{u_1,k}^n(\tilde u_1^n,\tilde u_2^n)
			\, d\tilde W^{n}_{u_1,k}}_{L^2(\Om)}^q\,\right]
		\\ &\quad
		\le C_T \tilde \E\left[\left(\int_0^T
		\sum_{k=1}^n\norm{\sigma^n_{u_1,k}(\tilde u_1^n,\tilde u_2^n)}_{L^2(\Om)}^2
		\dt\right)^{\frac{q}{2}}\right]\le C_{\sigma,T}.
	\end{align*}
	Therefore, an application of Vitali's convergence
	theorem, we deduce from \eqref{eq:stoch-conv-tmp1}
	$$
	\int_0^t \sigma_{u_1}^n(\tilde u_1^n,\tilde u_2^n) \, d\tilde W_{u_1}^{(n)}(s)
	\to
	\int_0^t \sigma_{u_1}(\tilde u_1,\tilde u_2) \, d\tilde W_{u_1} (s) \quad
	\text{in $L^2\left(\tilde D,\tilde \cF,\tilde P;
		L^2(0,T;L^2(\Om))\right)$}.
	$$
	Then, using this and the fact that $\Pi_n \vphi_{u_1}\to \vphi_{u_1}$
	in $L^2(\Om)$, we deduce
	\begin{align*}
		& \tilde \E \left[\int_0^T \En_Z(\omega,t) \left(
		\int_0^t \int_{\Om} \sigma_{u_1}^n(\tilde u_1^n,\tilde u_2^n)
		\Pi_n \vphi_{u_1}
		\dx\, d\tilde W_{u_1}^n(s)\right)\dt \right]
		\\ & \quad  = \tilde \E \left[\int_0^T
		\int_{\Om} \left(\int_0^t \sigma_{u_1}^n(\tilde u_1^n,\tilde u_2^n)
		\, d\tilde W_{u_1}^{(n)} (s)\right)
		\bigl(\En_Z(\omega,t)\Pi_n \vphi_{u_1}(x)\bigr)
		\dx\dt\right] \\ & \qquad
		\overn \tilde \E \left[\int_0^T \En_Z(\omega,t) \left(
		\int_0^t \int_{\Om} \sigma_{u_1}(\tilde u_1,\tilde u_2)
		\vphi_{u_1} \dx \, d\tilde W_{u_1} (s)\right)\dt\right].
	\end{align*}
	This concludes the proof of \eqref{eq:weakform-u-tilde}. The proof is the same \eqref{eq:weakform-v-tilde}.
\end{proof}


\section{Maximum principle of the solutions}\label{Sec8}
In this section we prove that the martingale
solution $(u_1,u_2)$ constructed as the limit of
the Faedo-Galerkin approximations $\left(u_1^n,u_2^n\right)$
is non-negative and bounded in $L^\infty$ almost surely.
In our proof ofe the lemma below, we write $a^-$
for the negative part, $\max(-a,0)$, of $a\in\R$.
Herein, we work with a smooth approximation
$S_\eps(\cdot)$ of $(\cdot)^-$.

The nonnegativity result is given by the following lemma
\begin{lem}
	The solution $(u_1,u_2)$ constructed in
	Theorem \ref{thm} is non-negative and bounded in $L^\infty$ almost surely.
\end{lem}

\begin{proof}
 For simplicity, we drop the tildes on the relevant functions,
	writing for example $u_i^n,u_i$ instead of $\tilde u_i^n,\tilde u_i$ for $i=1,2$.
	For $\eps>0$, denote by $S_\eps(w)$ the $C^2$
	approximation of $\left(w^-\right)^2$ defined by
	\begin{equation*}
		S_{\eps}(w)=
		\begin{cases}
			w^2-\frac{\eps^2}{6} & \text{if $w<-\eps$}, \\
			-\frac{w^4}{2\eps^2}
			-\frac{4w^3}{3\eps}
			& \text{if $-\eps \leq w<0$}, \\
			0 &\text{if $w\geq 0$}.
		\end{cases}
	\end{equation*}
	Note that
	\begin{equation*}
		S_{\eps}'(w)=
		\begin{cases}
			2w & \text{$w<-\eps$}, \\
			-\frac{2w^3}{\eps^2}
			-\frac{4w^2}{\eps}
			& \text{$-\eps \leq w<0$}, \\
			0 &\text{$w\geq 0$}
		\end{cases}
		\quad	S_{\eps}''(w)=
		\begin{cases}
			2 & \text{$w<-\eps$}, \\
			-\frac{6w^2}{\eps^2}
			-\frac{8w}{\eps}
			& \text{$-\eps \leq w<0$}, \\
			0 &\text{$w\geq 0$}.
		\end{cases}
	\end{equation*}
	Observe that $S_{\eps}(w)\ge 0$,
	$S_{\eps}'(w)\leq 0$, and $S_{\eps}''(w)\geq 0$
	for all $w\in \R$. Moreover, as $\eps \to 0$,
	the following convergences hold, uniformly in $w\in \R$:
	$S_{\eps}(w) \to \left(w^-\right)^2$,
	$S_{\eps}'(w)\to -2 w^-$, and
	$S_{\eps}''(w) \to
	\begin{cases}
		2 & \text{if $w< 0$}
		\\
		0 & \text{if $w\geq 0$}
	\end{cases}$.
	Now, an application of It\^{o} formula to
	$S_\eps(u_1^n)$, where $u_1^n$ solves
	\eqref{eq:approx-eqn-integrated}, gives
	\begin{equation}\label{eq:nonegativity1}
		\begin{split}
			&\int_\Om S_\eps(u_1^n(t))\dx
			-\int_\Om S_\eps(u_1^n(0))\dx
			\\ & \quad
			=- \int_0^t \int_{\Om}d_1\, S_{\eps}''(u_1^n(s))
			\abs{\Grad u_1^n}^2 \dx\ds
			+\int_0^t \int_{\Om}S_{\eps}''(u_1^n(s))
			\chi(u_1^n) \nabla u_2^n\cdot \Grad u_1^n \dx\ds
			\\ & \qquad \qquad
			+\int_0^t \int_{\Om}
			S_{\eps}'(u_1^n(s))
			F_1(u_1^n,u_2^n)\dx \ds
			+\sum^n_{k=1}\int_0^t \int_{\Om}
			S_{\eps}'(u_1^n(s))
			\sigma_{u_1,k}^n(u_1^n,u_2^n)\dx \dW_{u_1,k}^n
			\\ &\qquad\qquad\qquad
			+\frac{1}{2} \sum^n_{k=1} \int_0^t
			\int_{\Om}S_{\eps}''(u_1^n(s))
			\left(\sigma_{u_1,k}^n(u_1^n,u_2^n)\right)^2 dx \ds
			=:\sum_{i=1}^5 I_i.
		\end{split}
	\end{equation}
	It is easy to see that $I_1\le 0$.
	From condition \eqref{S3},
	\begin{equation}\label{eq:gamma-eps-prop}
		\begin{split}	
			& S_{\eps}''(w)=0 \quad
			\text{for $w\geq 0$},
			\quad
			\text{and}
			\quad
			S_{\eps}''(w)\geq 0
			\quad \text{for $w \in \R$},
			\\ & \text{and}
			\quad
			\chi(w)=0,
			\quad \text{for $w\leq 0$}.
		\end{split}
	\end{equation}
	Consequently $I_2=0$. Similarly, from the definition of the function
	$F_1$, cf.~\eqref{entries-positive},
	it follows that $I_3=0$.
	
	Using the convergences in \eqref{eq:tilde-conv} and
	sending $n\to \infty$ in \eqref{eq:nonegativity1}, we
	obtain
	\begin{equation}\label{eq:nonegativity1-new}
		\begin{split}
			&\E\left[\norm{S_\eps(u_1(t))}_{L^2(\Om)}^2\right]
			-\E\left[\norm{S_\eps(u_1(0))}_{L^2(\Om)}^2\right]
			\\ & \qquad
			\leq \E\left[\sum_{k=1}^{\infty}
			\int_0^t \int_{\Om} S_{\eps}''(u_1(t))
			\left(\sigma_{u_1,k}^n(u_1,u_2)\right)^2\dx \ds\right],
			\qquad t\in [0,T].
		\end{split}
	\end{equation}
	Next, we send $\eps \to 0$ in \eqref{eq:nonegativity1-new},
	and proceeding exactly as in
	\cite[Section 3.4]{Chekroun:2016aa}, to arrive at
	\begin{equation}\label{eq:nonegativity2}
		\E\left[\norm{u_1^-(t)}_{L^2(\Om)}^2\right]
		-\E\left[\norm{u_1^-(0)}_{L^2(\Om)}^2\right]
		\leq C\,  \E\left[ \int_0^t
		\norm{u_1^-(s)}_{L^2(\Om)}^2 \ds\right],
	\end{equation}
	for a.e.~$t\in [0,T]$ where $C>0$ is a constant.
	Finally, by the nonnegativity of $u_{1}(0)$ and applying Gronwall's
	inequality in \eqref{eq:nonegativity2}, we
	conclude that $u_1^-=0$ a.e.~in $(0,T)\times\Om$,
	almost surely. Along the same lines, it follows that
	$u_2\geq 0$ a.e.~in $(0.T)\times \Om$, almost surely.\\
	
	\noindent Now, the aim is to prove that the martingale solution $u_i$ is bounded by a number $M_i>0$ a.e. and a.s. for $i=1,2$. An application of It\^{o} formula to $S_\eps(M_1-u_1^n)$, we get
	\begin{equation}\label{eq:bound1}
		\begin{split}
			&\int_\Om S_\eps(M_1-u_1^n(t))\dx
			-\int_\Om S_\eps(M_1-u_1^n(0))\dx
			\\ & \quad
			=- \int_0^t \int_{\Om}d_1\, S_{\eps}''(M_1-u_1^n(s))
			\abs{\Grad u_1^n}^2 \dx\ds
			+\int_0^t \int_{\Om}S_{\eps}''(M_1-u_1^n(s))
			\chi(u_1^n) \nabla u_2^n\cdot \Grad u_1^n \dx\ds
			\\ & \qquad \qquad
			+\int_0^t \int_{\Om}
			S_{\eps}'(M_1-u_1^n(s))
			F_1(u_1^n,u_2^n)\dx \ds
			+\sum^n_{k=1}\int_0^t \int_{\Om}
			S_{\eps}'(M_1-u_1^n(s))
			\sigma_{u_1,k}^n(u_1^n,u_2^n)\dx \dW_{u_1,k}^n
			\\ &\qquad\qquad\qquad
			+\frac{1}{2} \sum^n_{k=1} \int_0^t
			\int_{\Om}S_{\eps}''(M_1-u_1^n(s))
			\left(\sigma_{u_1,k}^n(u_1^n,u_2^n)\right)^2 dx \ds
			=:\sum_{i=1}^5 \tilde I_i.
		\end{split}
	\end{equation}
	Observe that $\tilde I_1\le 0$.
	From \eqref{S3}, we obtain
	\begin{equation}\label{eq:gamma-eps-prop2}
		\begin{split}	
			& S_{\eps}''(M_1-w)=0 \quad
			\text{for $w\leq M_1$},
			\quad
			\text{and}
			\quad
			S_{\eps}''(M_1-w)\geq 0
			\quad \text{for $w \in \R$},
			\\ & \text{and}
			\quad
			\chi(w)=0,
			\quad \text{for $w\geq M_1$}.
		\end{split}
	\end{equation}
	As a result $\tilde I_2=0$. Similarly, from the definition of the function
	$F_1$, cf.~\eqref{reaction},
	it follows that $\tilde I_3=0$.
	
	Keeping in mind the convergences in \eqref{eq:tilde-conv}
	(see also \cite[Section 3.2]{Chekroun:2016aa}), we
	send $n\to \infty$ in \eqref{eq:nonegativity1} to
	arrive at the inequality:
	\begin{equation}\label{eq:nonegativity12-new}
		\begin{split}
			&\E\left[\norm{S_\eps(M_1-u_1(t))}_{L^2(\Om)}^2\right]
			-\E\left[\norm{S_\eps(M_1-u_1(0))}_{L^2(\Om)}^2\right]
			\\ & \qquad
			\leq \E\left[\sum_{k=1}^{\infty}
			\int_0^t \int_{\Om} S_{\eps}''(M_1-u_1(t))
			\left(\sigma_{u_1,k}^n(u_1,u_2)\right)^2\dx \ds\right],
			\qquad t\in [0,T].
		\end{split}
	\end{equation}
	Sending $\eps \to 0$ in \eqref{eq:nonegativity12-new}, we deduce
	\begin{equation}\label{eq:nonegativity22}
		\E\left[\norm{(M_1-u_1)^-(t)}_{L^2(\Om)}^2\right]
		-\E\left[\norm{(M_1-u_1)^-(0)}_{L^2(\Om)}^2\right]
		\leq C\,  \E\left[ \int_0^t
		\norm{(M_1-u_1)^-(s)}_{L^2(\Om)}^2 \ds\right],
	\end{equation}
	for a.e.~$t\in [0,T]$ where $C>0$ is a constant.
	Finally, since $u_1(0) \leq M_1$ and applying Gronwall's
	inequality in \eqref{eq:nonegativity22}, we
	conclude that $(M_1-u_1)^-=0$ a.e.~in $(0,T)\times\Om$,
	almost surely. Along the same lines, it follows that
	$u_2\leq M_2$ a.e.~in $(0.T)\times \Om$, almost surely.
	
\end{proof}

\section{Uniqueness of weak martingale solutions}\label{sec:uniq}
In this section we prove an $L^2$ stability estimate and consequently
a pathwise uniqueness result. We are now in a position to prove the stability result.

\begin{thm}\label{thm:uniq}
Assume \eqref{S3} and \eqref{eq:noise-cond} hold.
Let
$\bar U=\bigl(\cS, \bar u_1, \bar u_2\bigr)$ and
$\hat U=\bigl(\cS, \hat u_1, \hat u_2\bigr)$
be two weak solutions (according to Definition \ref{def:martingale-sol}),
relative to the same stochastic basis $\cS$, cf.~\eqref{eq:stochbasis}, with
initial data $\bar u_1(0)=\bar u_{1,0}$, $\hat u_1(0)=\hat u_{1,0}$,
$\bar u_2(0)=\bar u_{2,0}$, and $\hat u_2(0)=\hat u_{2,0}$, where
$\bar u_{1,0}, \hat u_{1,0}, \bar u_{2,0},\hat u_{2,0}\in
L^2\left(D,\cF,P;L^\infty(\Om)\right)$ and nonnegative.
There exists a positive constant $C\ge 1$ such that
\begin{equation}\label{eq:L2-stability}
	\begin{split}
		& \sum_{i=1,2} \E\left[\norm{\bar u_i- \hat u_i}_{L^2(\Om_T)}^2\right]
		\le C \sum_{i=1,2}
		\E \left [ \norm{\bar u_{i,0}-\hat u_{i,0}}_{L^2(\Om)}^2\right].
	\end{split}
\end{equation}
With $\bar u_{1,0}=\hat u_{1,0}$, $\bar u_{2,0}=\hat u_{2,0}$, it follows
that weak martingale solutions are unique.
\end{thm}

\begin{proof}
Set $u_1:=\bar u_1- \hat u_1$ and $u_2:=\bar u_2- \hat u_2$. We have $P$-a.s. for $i=1,2$,
\begin{align*}
	&u_i, \bar u_i, \hat u_i \in  L^\infty(\Om_T)\cap L^2((0,T);\tH(\Om))\cap L^\infty((0,T);L^2(\Om)).
\end{align*}

\noindent Subtracting the $(\tH(\Om))^*$ valued
equations for $\bar u_i, \hat u_i$ for $i=1,2$, we obtain
\begin{equation}\label{eq:uniq-v}
	\begin{split}
		&d u_1  -d_1\Delta u_1\dt+\Div (\chi(\bar u_1 )\nabla \bar u_2-\chi(\hat u_1 )\nabla \hat u_2)\dt=
		\left(F_1(\bar u_1,\bar u_2)-F_1(\hat u_1,\hat u_2)\right)\dt\\
		&\hskip7cm+\left(\sigma_{u_1}(\bar u_1,\bar u_2)-\sigma_{u_1}(\hat u_1, \hat u_2)\right) \dW_{u_1}(t),
		\\
		& d u_2  -d_2\Delta u_2\dt= \left(F_2(\bar u_1,\bar u_2)-F_2(\hat u_1,\hat u_2)\right)\dt+\left(\sigma_{u_2}(\bar u_1,\bar u_2)-\sigma_{u_2}(\hat u_1, \hat u_2)\right) \dW_{u_2}(t).
	\end{split}
\end{equation}
Now we define the function $\cN_w \in H^2(\Omega)\cap L^2(\Omega)$ such that
$\displaystyle \int_\Omega \cN_w\,dx =0$
and solution of the problem
\begin{equation}\label{ellip:uniq-gener}
-\Delta \cN_w=w \mbox{ in $\Omega$} \qquad \mbox{ and }\qquad
\frac{\partial  \cN_w}{\partial \eta}=0\mbox{  on
}\partial{\Omega}
\end{equation}
for a.e. $t \in (0,T)$. 
Multiplying the first equation in (\ref{eq:uniq-v}) by $\cN_{u_1}$, we obtain
\begin{equation}\label{est1:uniq-gener}
	\begin{split}
		 \left(d u_1, \cN_{u_1}\right)&= d_1 \left(u_1,\Delta \cN_{u_1}\right)\dt
                 -\bigl (\chi(\bar u_1 )\nabla \bar u_2-\chi(\hat u_1 )\nabla \hat u_2 , \Grad \cN_{u_1}\bigr)\dt
		\\ & \qquad
		+ \left (F_1(\bar u_1,\bar u_2)-F_1(\hat u_1,\hat u_2), \cN_{u_1} \right)\dt
		+ \sum_{k=1}^n\left( \sigma_{u_1,k}(\bar u_1,\bar u_2)-\sigma_{u_1,k}(\hat u_1, \hat u_2),
		\cN_{u_1} \right)\dW_{u_1,k}(t)\\
                  &= d_1 \left(u_1,\Delta \cN_{u_1}\right)\dt
		-\bigl ((\chi(\bar u_1 )-\chi(\hat u_1 ))\nabla \bar u_2 , \Grad \cN_{u_1}\bigr)\dt
                   -\bigl (\chi(\hat u_1 ) \nabla u_2, \Grad \cN_{u_1}\bigr)\dt
		\\ & \qquad
		+ \left (F_1(\bar u_1,\bar u_2)-F_1(\hat u_1,\hat u_2), \cN_{u_1} \right)\dt
		+ \sum_{k=1}^n\left( \sigma_{u_1,k}(\bar u_1,\bar u_2)-\sigma_{u_1,k}(\hat u_1, \hat u_2),\cN_{u_1} \right)\dW_{u_1,k}(t).
	\end{split}	
\end{equation}

Now, using \eqref{ellip:uniq-gener} to deduce
\begin{equation}\label{est2:uniq-gener}
\begin{array}{lll}
\displaystyle 2\int_0^t\left(d u_1, \cN_{u_1}\right)&\displaystyle= -2\int_0^t \left(d\Delta \cN_{u_1},\cN_{u_1}\right) \\
&\displaystyle= \int_0^t d\abs{\Grad \cN_{u_1}}^2\\
&\displaystyle=\int_\Omega |\nabla \cN_{u_1}(t)|^2 \,dx  -\int_\Omega \abs{\nabla \cN_{u_1}(0)}^2 \,dx           \\
&=\displaystyle\int_\Omega |\nabla \cN_{u_1}(t)|^2 \,dx.
  \end{array}
\end{equation}
Integrating over $\Om_t$ and using the H\"{o}lder's, Young's, Sobolev poincar\'e's and Burkholder-Davis-Gundy inequalities \eqref{eq:noise-cond2} yields from (\ref {est1:uniq-gener})
\begin{eqnarray}\label{est4:uniq-gener}
  \begin{array}{ll}\displaystyle
\displaystyle \int_0^t(d u_1,\cN_{u_1} ) &\le
-d_1 \displaystyle \iint_{\Om_t}|u_1|^2 \,dx\,ds
+\tau\iint_{\Om_t}|u_1|^2\,dx\,ds
\\
&\quad +C\displaystyle \int_0^T \left\|{\nabla \bar u_2}\right\|^2_{L^\infty(\Omega)} \left\|{\nabla \cN_{u_1}}\right\|_{L^2(\Omega)}^2 \,ds\\
&\quad +\displaystyle\frac{d_2}{2}\int_{0}^T\left\|{\nabla
u_2}\right\|_{L^2(\Omega)}^2\,ds
+C\int_{0}^T\left\|{\nabla \cN_{u_1}}\right\|_{L^2(\Omega)}^2 \,ds\\
 &\quad +\displaystyle\tau\iint_{\Om_t}|u_1|^2 \,dx\,ds
+C\displaystyle\int_{0}^T\left\|{\nabla \cN_{u_1}}\right\|_{L^2(\Omega)}^2 \,ds\\
&\quad +\displaystyle\tau \iint_{\Om_t}|u_1|^2 \,dx\,ds+C\iint_{\Om_t}|u_2|^2\,dx\,ds\\
&\quad +C\displaystyle\int_{0}^T\left\|{\nabla \cN_{u_1}}\right\|_{L^2(\Omega)}^2 \,ds\\
&=\displaystyle (3\tau -d_1) \iint_{\Om_t}|u_1|^2\,dx\,ds\\
&\quad\displaystyle +C\int_{0}^T\left\|{\nabla
\bar u_2}\right\|^2_{L^\infty(\Omega)}
\left\|{\nabla \cN_{u_1}}\right\|_{L^2(\Omega)}^2 \,ds \\
&\quad+\displaystyle\frac{d_2}{2}\int_{0}^T\left\|{\nabla
u_2}\right\|_{L^2(\Omega)}^2\,ds
+3C \int_{0}^T\left\|{\nabla \cN_{u_1}}\right\|_{L^2(\Omega)}^2 \,ds\\
&\quad\displaystyle+C\iint_{\Om_t}|u_2|^2 \,dx\,ds,
  \end{array}
\end{eqnarray}
for some constant $C>0$.
An application of the It\^{o} formula to \eqref{eq:uniq-v} and H\"{o}lder's, Young's  inequalities and \eqref{eq:noise-cond2}, we obtain the following inequality:
\begin{equation}\label{eq:Ito-L2-v-diff_uniq1}
	\begin{split}
&\frac12\norm{u_2(t)}_{L^2(\Om)}^2 + d_2\int_0^t \int_{\Om} \abs{\Grad u_2}^2 \dx \ds
		\\ &
		\, \le \frac12\norm{u_2(0)}_{L^2(\Om)}^2
		+ \int_0^t \int_{\Om} \left(F_2(\bar u_1,\bar u_2)-F_2(\hat u_1, \hat u_2)\right)u_2\dx \ds
		\\ &  \quad
		+\sum_{k\ge 1}\int_0^t \int_{\Om}  \abs{\sigma_{u_2,k}(\bar u_1,\bar u_2)-\sigma_{u_2,k}(\hat u_1,\hat u_2)}^2 \dx\ds
		+\sum_{k\ge 1}\int_{\Om} u_2\left(\sigma_{u_2,k}(\bar u_1,\bar u_2)-\sigma_{u_2,k}(\hat u_1,\hat u_2)\right)\dx \dW_{u_2}^k\\
& \le \frac12\norm{u_2(0)}_{L^2(\Om)}^2+ \tau \iint_{\Om_t}|u_1|^2 \,dx\,ds
+C\iint_{\Om_t}|u_2|^2 \,dx\,ds+ C \int_{0}^T\left\|{\nabla \cN_{u_1}}\right\|_{L^2(\Omega)}^2 \,ds,
	\end{split}
\end{equation}
for some constant $C>0$. The consequence of (\ref {est4:uniq-gener}) and (\ref{eq:Ito-L2-v-diff_uniq1}) is
\begin{eqnarray}\label{est6:uniq-gener}
  \begin{array}{l}
\displaystyle \E\left[\norm{u_2(t)}_{L^2(\Om)}^2\right]
+\E\left[\norm{\nabla \cN_{u_1}(t,x)}_{L^2(\Om)}^2\right]          \\
\qquad \qquad \quad \displaystyle \le C \int_{0}^T\E\left[\Bigl(\left\|{\nabla \overline{u}_2}\right\|_{L^\infty(\Omega)}
+1\Bigl)\left\|{\nabla \cN_{u_1}(s)}\right\|_{L^2(\Omega)}^2 \right]\,ds +C\int_0^T E\left[\norm{u_2(s)}_{L^2(\Om)}^2\right]\,ds,
  \end{array}
\end{eqnarray}
for some constant $C>0$. Finally, the Gr\"{o}nwall lemma delivers from (\ref {est6:uniq-gener})
$$
u_2=0 \mbox{ and } \nabla \cN_{u_1}=0 \mbox{ a.e.~in $\Om_t$, almost surely},
$$
ensuring the uniqueness of weak martingale solutions.

\end{proof}

\end{document}